\documentclass
[11pt]{article}

\usepackage[a4paper, margin=1in]{geometry}
\usepackage{amsfonts,amsmath,amssymb,amsthm,graphicx,fixmath,latexsym,color}
\usepackage{xcolor}
\usepackage{hyperref,epsfig}
\usepackage{algorithmic}
\usepackage{algorithm}
\usepackage{xspace}

\providecommand{\remove}[1]{}

\newcommand{\mn}[0]{\medskip\noindent}

\newcommand{\F}{\mathcal{F}}
\newcommand{\B}{\mathcal{B}}
\newcommand{\M}{\mathcal{M}}
\newcommand{\h}{\mathcal{H}}

\newtheorem{theorem}{Theorem}[section]

\newtheorem{proposition}[theorem]{Proposition}
\newtheorem{claim}[theorem]{Claim}

\newtheorem{definition}[theorem]{Definition}
\newtheorem{remark}[theorem]{Remark}

\newtheorem{observation}[theorem]{Observation}
\newtheorem*{theorem*}{Theorem}
\newtheorem*{lemma*}{Lemma}
\newtheorem*{proposition*}{Proposition}

\begin{document}

\title{Blockers for Simple Hamiltonian Paths in Convex Geometric Graphs of Odd Order}

\author{Chaya Keller\thanks{Department of Mathematics, Ben-Gurion University of the NEGEV, Be'er-Sheva, Israel. \texttt{kellerc@math.bgu.ac.il}. Research partially supported by Grant 635/16 from the Israel Science Foundation, by the Shulamit Aloni Post-Doctoral Fellowship of the Israeli Ministry of Science and Technology, by the Kreitman Foundation Post-Doctoral Fellowship and by the Hoffman Leadership and Responsibility Program of the Hebrew University.}
\mbox{ }
and Micha A. Perles\thanks{Einstein Institute of Mathematics, Hebrew University, Jerusalem, Israel.
\texttt{perles@math.huji.ac.il}}
}

\maketitle

\begin{abstract}
Let $G$ be a complete convex geometric graph, and let $\mathcal{F}$ be a family of subgraphs of $G$. A \emph{blocker} for $\mathcal{F}$ is a set of edges, of smallest possible size, that has an edge in common with every element of $\mathcal{F}$. In~\cite{KP16} we gave an explicit description of all blockers for the family of simple (i.e., non-crossing) Hamiltonian paths (SHPs) in $G$ in the `even' case $|V(G)|=2m$. It turned out that all the blockers are simple \emph{caterpillar trees} of a certain class. In this paper we give an explicit description of all blockers for the family of SHPs in the `odd' case $|V(G)|=2m-1$. In this case, the structure of the blockers is more complex, and in particular, they are not necessarily simple. Correspondingly, the proof is more complicated.
\end{abstract}

\section{Introduction}

In this paper we consider convex geometric graphs, i.e., graphs whose vertices are points in convex position in the plane, and whose edges are segments connecting pairs of vertices. Throughout the paper, $G$ will denote a convex geometric graph and $\F$ will denote a family of subgraphs of $G$.
\begin{definition}
A set of edges in $E(G)$ is called a blocking set for $\F$ if it intersects (i.e., contains an edge of) every element of $\F$. A \emph{blocker} for $\F$ is a blocking set of smallest possible size. The family of blockers for $\F$ is denoted $\B(\F)$.
\end{definition}
Let $G=CK(n)$ be the complete convex geometric graph of order $n$. Finding the size of the blockers $\B(\F)$ for a family $\F$ is a natural Tur\'{a}n-type question, as it is clearly equivalent to the question: what is the maximal possible number of edges in a convex geometric graph on $n$ vertices that does not include any element of $\F$ as a subgraph?
This question was extensively studied with respect to various families $\F$, e.g., all sets of $k$ disjoint edges~\cite{K79,KP96} and all sets of $k$ pairwise crossing edges~(\cite{CP92}, and see also~\cite{BKV03}).

In various cases, including the two cases mentioned above, the size of the blockers is known. In these cases, the next natural step is to provide a \emph{characterization} of the blockers for $\F$.

\medskip

In~\cite{KP12} we considered the `even' case $G=CK(2m)$, and provided a complete characterization of the blockers for the family $\M$ of simple (i.e., non-crossing) perfect matchings (SPMs) of $G$. We described the blockers as certain simple \emph{caterpillar} subtrees of $G$ of size $m$. (Roughly speaking, a caterpillar is a tree whose derivative is a path. See~\cite{Caterpillar1} for the exact definition of caterpillars.) In~\cite{KP16} we showed that the blockers for the family $\h$ of simple Hamiltonian paths (SHPs) in $CK(2m)$ are exactly the same as the blockers for SPMs.

In this paper we consider the (somewhat more complicated) case of SHPs in a complete convex geometric graph of odd order $2m-1$. Our main result is a complete description of the blockers for $\h$. This time, not all blockers are caterpillar trees, and they do not even have to be simple. In order to describe the blockers, we have to define a few additional notions.

\medskip

Let $V$ be the set of vertices of $G=CK(2m-1)$ (viewed as the vertex set of a convex polygon $P$ in the plane), labelled cyclically from $0$ to $2m-2$. The \emph{distance} between two vertices $i,j$ is $\min(|i-j|,(2m-1)-|i-j|)$. The edges that belong to the boundary of $P$ (i.e., the edges of the form  $[i,i+1]$ ($0 \leq i \leq 2m-3$), along with the edge $[2m-2,0]$) are called \emph{boundary edges} of $G$. The \emph{direction} of an edge $[i,j]$ is defined to be $i+j (\bmod (2m-1))$.
 Two edges are called \emph{parallel} if they have the same direction. Note that every edge of $G$ is parallel to a unique boundary edge.
 Two directions are called \emph{boundary-consecutive} if they contain consecutive boundary edges. (i.e., if they differ by 2 ($\bmod(2m-1)$).)

\medskip

Our main theorem is the following:

\medskip \noindent \textbf{Theorem.} Let $G=CK(2m-1)$, and let $\h$ be the family of simple Hamiltonian paths in $E(G)$. Any blocker for $\h$ consists of $m$ edges in $m$ boundary-consecutive directions, and up to cyclical rotation by $0 \leq k \leq 2m-2$, it has one of the following two forms. Moreover, all the sets described below (Class A and Class B) are indeed blockers.

\medskip

\begin{figure}[tb]
\begin{center}
\scalebox{1.0}{
\includegraphics{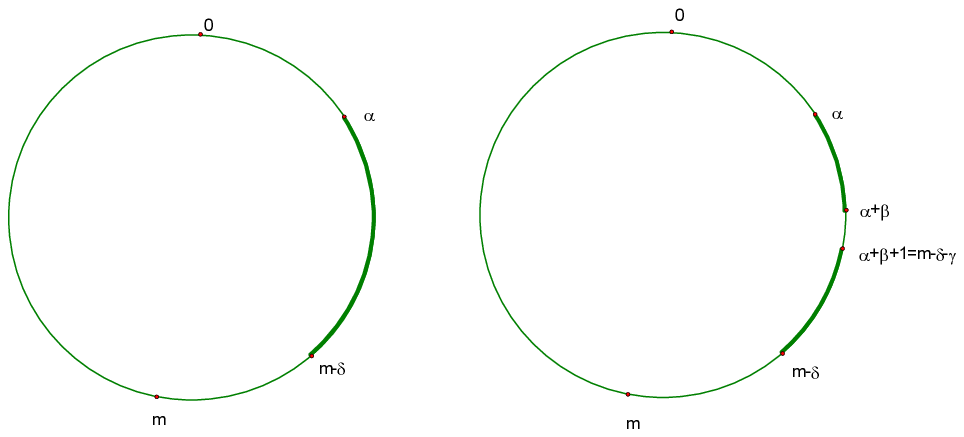}
} \caption{The notations used in the descriptions of the blockers in the main theorem. The left figure corresponds to a blocker of Class A, and the right figure corresponds to a blocker of Class B.}
\label{Fig:Aux-Intro1}
\end{center}
\end{figure}

\noindent \textbf{Class A. Blockers that contain a consecutive boundary path.} The edges of the blocker are parallel to the boundary edges $[0,1],[1,2],\ldots,[m-1,m]$. They consist of three parts:
\begin{enumerate}
\item The boundary path $BP=\langle \alpha, \alpha+1,\ldots,m-\delta \rangle$, for some $\alpha,\delta \geq 0$ with $0 \leq \alpha+\delta \leq m-2$.
The length of $BP$ is $m-\alpha-\delta$, and ranges between 2 and $m$.
This path is illustrated in the left part of Figure~\ref{Fig:Aux-Intro1}.

\item The edges $u_i=[i-1-\epsilon_i,i+\epsilon_i]$, $1 \leq i \leq \alpha$ (where indices are taken modulo $2m-1$), for $\epsilon_1>\epsilon_2 \ldots > \epsilon_{\alpha}>0$, $\alpha - i +1 \leq \epsilon_i \leq m- \delta - i -1$. (These are the edges parallel to $[0,1],[1,2],\ldots,[\alpha-1,\alpha]$).


\item The edges $v_j=[m-j-\xi_j,m-j+1+\xi_j]$, $1 \leq j \leq \delta$, for $\xi_1>\xi_2 \ldots > \xi_{\delta}>0$, $\delta +1-j \leq \xi_i \leq m-j- \alpha -1$. (These are the edges parallel to $[m-\delta,m-\delta+1],\ldots,[m-1,m]$),
\end{enumerate}
where in addition it is required that $\epsilon_1 + \xi_1 \leq m-2$ (which means that all edges of the second part lie `above' all edges of the third part).

Informally, the conditions mean that the edges of the second and third parts are in one-to-one correspondence with the boundary edges in the path $\langle 0,1,\ldots,m \rangle$ that are not included in the blocker. In addition, these edges are diagonals of the convex polygon $Q=\mathrm{conv}(V(G))$, and each one connects an internal vertex of $BP$ with a vertex that is not on $BP$. If we order them $u_{\alpha},u_{\alpha - 1}, \ldots, u_1,v_1,v_2,  \ldots, v_{\delta}$, then their `root' on $BP$ advances (weakly) along $BP$, and their directions decrease through the $\alpha + \delta$ values $2 \alpha -1, 2 \alpha -3, \ldots, 1,0,-2, \ldots, -2 \delta +2 $ ($\alpha$ positive odd values followed by $\delta$ even non-positive values).
Any two such edges are disjoint, and the distance between the other endpoints is larger than the distance between the endpoints on the $BP$. Blockers of this class are simple caterpillars, and are actually similar to the blockers in the `even' case ($G=CK(2m)$). An example of a blocker of this class is presented in the left part of Figure~\ref{Fig:Blockers1}.

\medskip

\noindent \textbf{Class B. Blockers with a broken boundary path.} The edges of the blocker are parallel to the boundary edges $[0,1],[1,2],\ldots,[m-1,m]$. They consist of five parts, as follows:

\medskip \noindent {\rm (1)--(2).} The boundary paths $\langle \alpha, \alpha+1,\ldots,\alpha+\beta \rangle$ and $\langle \alpha+\beta+1,\alpha+\beta+2,\ldots, m-\delta \rangle$ of lengths $\beta,\gamma$, respectively, for some $\alpha,\delta \geq 0$, $\beta,\gamma \geq 2$ with $\alpha+\delta \leq m-5$ and $\beta+\gamma=m-\alpha-\delta-1$. (That is, the boundary path $\langle \alpha, \alpha+1,\ldots, m-\delta \rangle$ misses a single edge $[\alpha+\beta,\alpha+\beta+1]$.) These two paths are illustrated in the right part of Figure~\ref{Fig:Aux-Intro1}.

\medskip \noindent {\rm (3).} The edge $[\alpha+\beta-\eta,\alpha+\beta+1+\eta]$, for some $1 \leq \eta \leq \min(\beta-1,\gamma-1)$ (which is parallel to the `missing' boundary edge $[\alpha+\beta,\alpha+\beta+1]$).

\medskip \noindent {\rm (4).} The edges $[i-1-\epsilon_i,i+\epsilon_i]$, $1 \leq i \leq \alpha$ (where indices are taken modulo $2m-1$), for $\alpha+\beta-1>\epsilon_1>\epsilon_2 \ldots > \epsilon_{\alpha}>0$. (These are the edges parallel to $[0,1],[1,2],\ldots,[\alpha-1,\alpha]$; compared to Case A, we have the additional condition that they lie `above' the missing boundary edge).

\medskip \noindent {\rm (5).} The edges $[m-j-\xi_j,m-j+1+\xi_j]$, $1 \leq j \leq \delta$, for $\gamma+\delta-1>\xi_1>\xi_2 \ldots > \xi_{\delta}>0$. (These are the edges parallel to $[m-\delta,m-\delta+1],\ldots,[m-1,m]$; compared to Case A, we have the additional condition that they lie `below' the missing boundary edge).

Blockers of this class are not caterpillars, and moreover, they are not necessarily simple. An example of a blocker of this class is presented in the right part of Figure~\ref{Fig:Blockers1}, and another example -- which is not even simple -- is presented in Figure~\ref{Fig:Non-Simple-Blocker}.

\begin{figure}[tb]
\begin{center}
\scalebox{1.0}{
\includegraphics{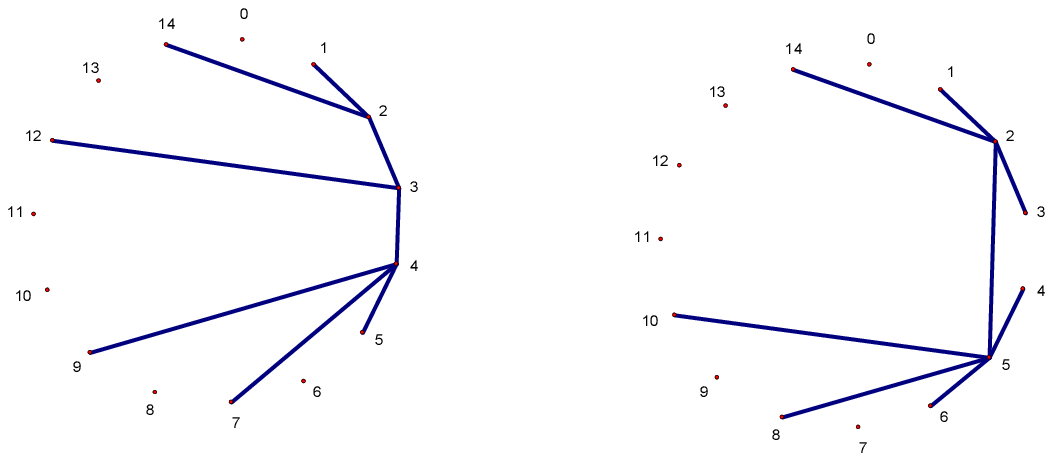}
} \caption{Two blockers for SHPs in $CK(15)$. The left blocker is of Class A, with parameters $\alpha=1$, $\delta=3$, $\epsilon_{1}=1$, $(\xi_1,\xi_2,\xi_3) = (4,2,1)$. The right blocker is of Class B, with parameters $(\alpha,\beta,\gamma,\delta)=(1,2,2,2)$, $\eta=1$, $\epsilon_1=1$, and $(\xi_1,\xi_2)=(2,1)$.}
\label{Fig:Blockers1}
\end{center}
\end{figure}

\begin{figure}[tb]
\begin{center}
\scalebox{1.0}{
\includegraphics{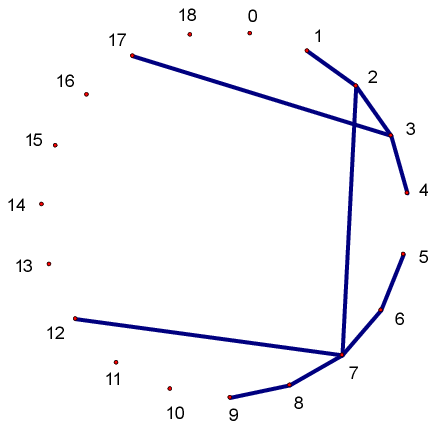}
} \caption{A blocker for SHPs in $CK(19)$ which is not even simple. Its parameters are  $(\alpha,\beta,\gamma,\delta)=(1,3,4,1)$, $\eta=2$, $\epsilon_1=2$, and $\xi_1=2$.}
\label{Fig:Non-Simple-Blocker}
\end{center}
\end{figure}

\medskip

The proof consists of three steps, whose formal description appears in the following sections. In Section~\ref{sec:boundary} we prove that the boundary edges of any blocker $B$ are either consecutive or consist of two runs of consecutive edges separated by a single edge that does not belong to the blocker (see Figure~\ref{Fig:Blockers1}). In Section~\ref{sec:Class_A} we prove that any blocker has to satisfy all other conditions that are common to Classes~A and B. In Section~\ref{sec:Class_B} we prove that if the boundary path of a blocker is broken (i.e., misses an edge) then that blocker must satisfy the additional conditions given in the definition of Class~B above. Finally, in Section~\ref{sec:converse} we prove the converse direction of the theorem, namely, that all elements of Classes~A and B are indeed blockers.

\section{Notations and Basic Observations}

The vertices of our complete convex geometric graph are $2m-1$ points in convex position, whose convex hull is a convex $(2m-1)$-gon $P$ ($m \geq 2$). For the sake of convenience we assume, without loss of generality, that $P$ is a regular $(2m-1)$-gon. We label the vertices (clockwise) $0,1,2,\ldots,2m-2$, with boundary edges $[i-1,i]$ ($0 \leq i < 2m-1$) and $[2m-2,0]$. We regard the labels as elements of the cyclic group $\mathbb{Z}_{2m-1}=\mathbb{Z}/(2m-1)\mathbb{Z}$. We define the \emph{direction} of an edge $[a,b]$ of $G$ to be the modular sum $a+b(\bmod(2m-1))$. Two edges are \emph{parallel} iff they have the same direction. For each $i \in \mathbb{Z}_{2m-1}$, the set $D_i$ of parallel edges in direction $i$ is a simple almost perfect matching that misses only one vertex.

\medskip

Our first observation is that the size of the blockers for SHPs is $m$. It is proved as follows.
\begin{proposition}
Let $B$ be a blocking set for all SHPs of $G$. Then $|B| \geq m$.
\end{proposition}

\begin{proof}
The union of two adjacent $D$'s, say $D_i$ and $D_{i+1}$ ($i \in \mathbb{Z}_{2m-1}$) forms a zig-zag Hamiltonian path.
This implies that a blocker for SHPs in $CK(2m-1)$ cannot miss two adjacent sets $D_i,D_{i+1}$ ($i \in \mathbb{Z}_{2m-1}$). Due to the circular structure of $\mathbb{Z}_{2m-1}$, it follows that there is no blocker of size $<m$.
\end{proof}

\begin{proposition}
There exists a blocking set $B$ for all SHPs of $G$ with $|B|=m$.
\end{proposition}

\begin{proof}
Let $B$ be the set of edges of the boundary path $\langle 0,1,\ldots,m \rangle$. The reader will easily convince himself that if $C$ is any spanning simple subgraph of $CK(2m-1)$ without isolated vertices, and if $[a,c] \in E(C)$ for some $0 \leq a < c \leq m$, then $C$ must use some boundary edge $[b,b+1]$, with $a \leq b <b+1 \leq c$. Thus, if $C$ avoids $B$, then every edge of $C$ must use at least one of the $m-2$ vertices $m+1,m+2,\ldots,2m-2$. This means that if $C$ is a simple path, then it can have at most $2(m-2)$ edges. But an SHP in $CK(2m-1)$ has $2(m-1)$ edges, a contradiction.
\end{proof}

Combining these two propositions, we find that the size of blockers for SHPs in $G$ is $m$. Furthermore, it follows that any blocker must visit the sets $D_i$ alternately, i.e., for some initial value $a$ it must visit $D_a,D_{a+2},\ldots,D_{a+2i},\ldots,D_{a+2m-2}$. (Note that the directions $a$ and $a+2m-2$ are adjacent.)
The directions $a,a+2,\ldots,a+2i,\ldots,a+2m-2$ are exactly the directions of the edges of the boundary path $\langle c,c+1,\ldots,c+m-1,c+m \rangle$ of $P$, where $2c \equiv a-1 (\bmod (2m-1))$.
Due to the circular symmetry of $P$, we may restrict our attention to the case $c=0$.

We conclude this section with the following simple observation:
\begin{observation}
\label{2-bdry-edges}
Any blocker $B$ contains at least two boundary edges.
\end{observation}
Indeed, the complement of a single edge on the boundary of $P$ is an SHP.

\section{The Boundary Edges of a Blocker}
\label{sec:boundary}

In this section we show that the boundary edges of blockers are of a very specific structure. Recall that in the `even' case (i.e., $G=CK(2m)$), the boundary edges of a blocker form a single path on the boundary of $\mathrm{conv}(V(G))$. While this is not necessarily the case for $CK(2m-1)$, we show that in this case, the boundary edges of any blocker consist either of a single path, or of two paths with a single missing edge in between.

\medskip

Suppose $B$ is a blocker for SHP's in $CK(2m-1)$. We already know that $B$ consists of $m$ edges of different directions, and that the directions of the edges of $B$ are those of a boundary path of length $m$. Assume, w.l.o.g., that this is the path $\langle 0,1, \ldots m \rangle$. We call this path \emph{the directional support of $B$}. Denote by $e_i$ $(i=1, \ldots , m)$ the edge of $B$ parallel (or equal) to the boundary edge $[i-1,i]$. Recall that by Observation~\ref{2-bdry-edges}, $B$ must contain at least two boundary edges. Assume $[\alpha, \alpha+1]$ and $[m-\delta -1, m-\delta]$ are the first and the last edges of $B$ on the directional support  $\langle 0,1, \ldots m \rangle$. $(0 \leq \alpha<\alpha+1 \leq m-\delta -1<m-\delta \leq m)$. We call the boundary path $\langle \alpha,\alpha+1, \ldots m - \delta \rangle$ the $\emph{backbone}$ of $B$.

\begin{proposition}\label{Prop:Boundary}
$B$ misses at most one edge of its backbone.
\end{proposition}

\begin{proof}
Suppose, on the contrary, that $B$ misses two edges of its backbone: $h_1=[\beta', \beta'+1]$ and $h_2=[\gamma', \gamma'+1]$ (where $\alpha < \beta' < \gamma' < m-\delta -1$), and possibly some other edges as well. Note that this is possible only if $\alpha + \delta \leq m-4$.
We define $\alpha + \delta +1$ SHP's $P_\nu$, for $-\delta \leq \nu \leq \alpha$, with the following properties:
\begin{enumerate}
\item These $P_{\nu}$'s use none of the edges $e_{\alpha+1},e_{\alpha+2},\ldots,e_{m-\delta}$ (i.e., the edges of $B$ which either belong to the backbone of $B$ or are parallel to an edge in the backbone of $B$).

\item Each of the $\alpha+\delta$ remaining edges of $B$, i.e., $e_1,\ldots,e_\alpha$ and $e_{m-\delta +1},\ldots,e_m$, blocks at most a single $P_\nu$.
\end{enumerate}
By the pigeonhole principle, this will imply that at least one $P_\nu$ is not blocked by $B$, hence $B$ is not a blocker.

\medskip The SHP $P_{\nu}$ is composed of five sets of edges, as follows:
\begin{itemize}
\item Initial Section: $\langle \nu, \nu-1, \nu+1, \nu-2, \nu+2, \ldots, \beta' \rangle$. This section consists of $2 (\beta' - \nu)$ edges, of directions $2\nu-1$ and $2\nu$ alternately.

\item First Junction: $h_1=[\beta',\beta'+1]$.

\item Middle Section: $\langle \beta'+1,2\nu-\beta'-1,\beta'+2,\ldots,\gamma' \rangle$. This section consists of $2 (\gamma'-\beta' - 1)$ edges, of directions $2\nu$ and $2\nu+1$ alternately. Note that this section is empty if $\gamma'=\beta'+1$, i.e., if the two holes are consecutive.

\item Second Junction: $h_2=[\gamma',\gamma'+1]$.

\item Terminal Section: $\langle \gamma'+1,2\nu-\gamma',\gamma'+2,\ldots,\nu+m+1,\nu+m \rangle$. This section consists of $2 (\nu+m-\gamma'-1)$ edges, of directions $2\nu+1$ and $2\nu+2$ alternately.
\end{itemize}
Four examples of $P_{\nu}$ are presented in Figure~\ref{Fig:4-SHP's-2-holes}.

\medskip We denote by $I_\nu$ the main diagonal $[\nu, \nu +m]$ of $P$. (The labels of the vertices are elements of $\mathbb{Z}_{2m-1}$, so $\nu$ is replaced by $2m-1+\nu$ for $-\delta \leq \nu < 0$).
The \emph{major side of $I_\nu$} is the closed half-plane bounded by $\mathrm{aff}(I_\nu)$ that includes the boundary path $\langle \nu, \nu+1, \ldots, \nu+m \rangle$ of length $m$. The \emph{minor side of $I_\nu$} is the complementary closed half-plane that includes the boundary path $\langle \nu+m, \nu+m+1, \ldots, \nu+2m-1 \rangle$ of length $m-1$. Note that as $\nu$ increases from $- \delta$ to $\alpha$, $\nu + m$ increases from $m- \delta$ to $m+ \alpha$. ($m+ \alpha \leq 2m-1-\delta -3$, since $\alpha + \delta \leq m-4$). It follows that the backbone $\langle \alpha, \ldots, m- \delta \rangle$ of $B$ lies on the major side of $I_\nu$ for all $-\delta \leq \nu \leq \alpha$. More precisely, all points of the backbone except possibly the endpoints $\alpha$ and $m-\delta$, lie in the interior of the major side of $I_\nu$. The two boundary edges of $P_\nu$, other than $h_1$ and $h_2$, namely $[\nu-1,\nu]$ and $[\nu+m,\nu+m+1]$, lie on the minor side of $I_\nu$.

\begin{claim}
The SHP's $P_{-\delta}, P_{-\delta+1},\ldots,P_{\alpha}$ satisfy the aforementioned Properties~1 and~2.
\end{claim}

\medskip \noindent {\it Proof of~(1).} For each $\nu$, the directions of the edges of $P_{\nu}$, except for the edges $h_1,h_2$, are $2\nu-1,2\nu,2\nu+1,2\nu+2$.
These are the directions of the four boundary edges incident to the endpoints of $I_{\nu}$, namely: $[\nu-1,\nu],[\nu,\nu+1],[\nu+m-1,\nu+m],[\nu+m,\nu+m+1]$.
Suppose, on the contrary, that some edge $e$ of $P_{\nu}$ belongs to $B$ and is either equal or parallel to an edge of the backbone $\langle \alpha, \ldots, m-\delta \rangle$ of $B$.

\medskip \noindent {\it Case~(I) -- $e$ is a boundary edge.} Then $e$ is neither $h_1$ nor $h_2$, since $h_1$ and $h_2$ are not in $B$. Thus $e$ is a boundary edge that lies on the minor side of $I_{\nu}$ and is incident to one of the two endpoints of $I_{\nu}$. In particular, $e$ is not an edge of the backbone, and is not parallel to an edge of the backbone. (Two distinct boundary edges of an odd-sided polygon are never parallel.)

\medskip \noindent {\it Case~(II) -- $e$ is a diagonal of $\mathrm{conv}(V(G))$.} Then $e$ is parallel (not equal) to some edge $[\iota,\iota+1]$, $\alpha \leq \iota <m-\delta$, of the backbone. Since $e \in B$, necessarily $[\iota,\iota+1] \notin B$, i.e., $[\iota,\iota+1]$ is a `hole' in the backbone ($h_1$, $h_2$, or possibly a third hole). It follows that $[\iota,\iota+1]$ is an internal edge of the backbone, i.e., $\alpha +1 \leq \iota <m-\delta -1$. Thus $[\iota,\iota+1]$ lies in the interior of the major side of $I_{\nu}$. But $e$, being an edge of $P_{\nu}$ other than $h_1$ and $h_2$, is also parallel to a boundary edge that is incident with an endpoint of $I_{\nu}$. This is clearly impossible.


\begin{figure}[tb]
\begin{center}
\scalebox{1.0}{
\includegraphics{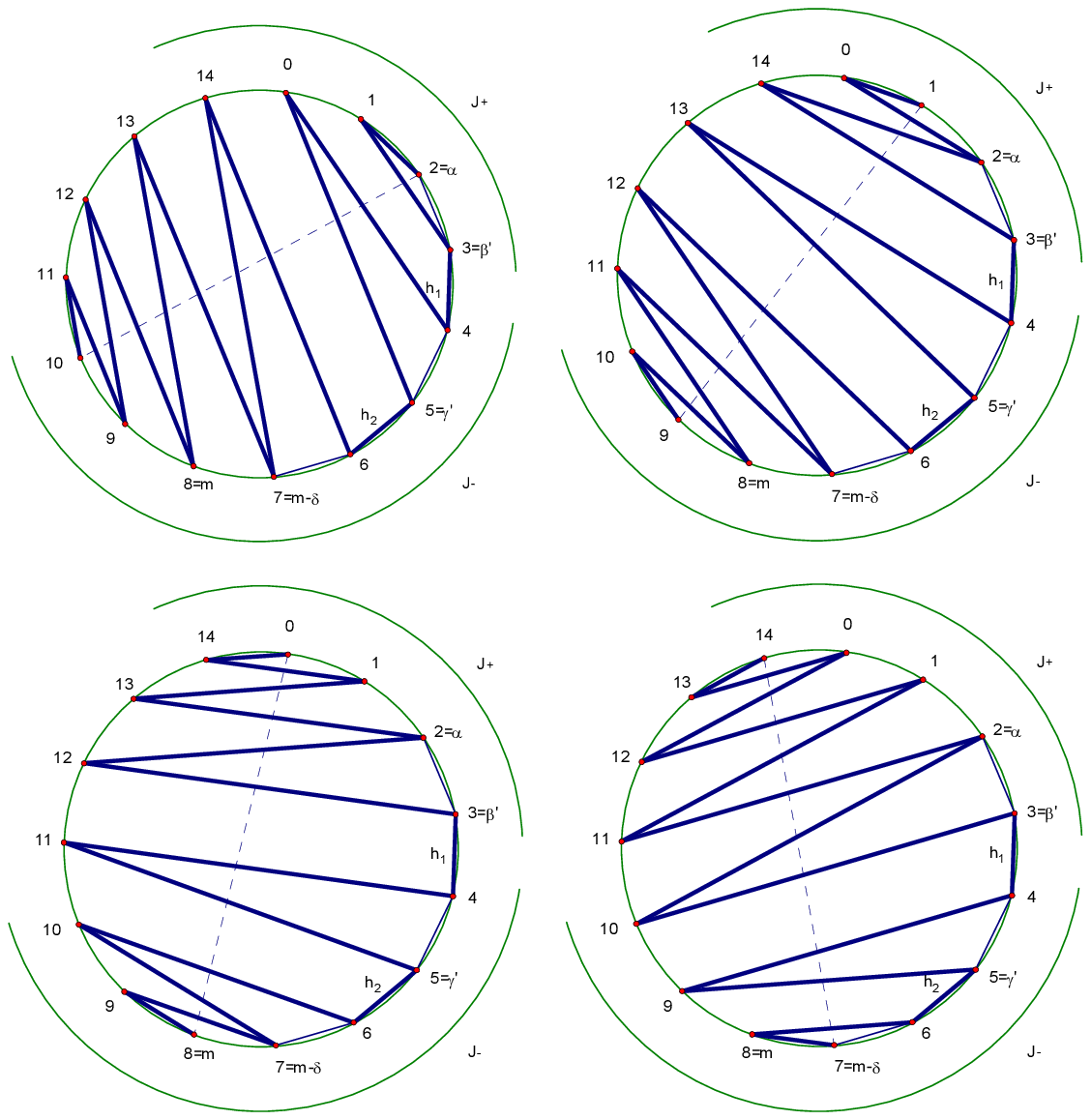}
} \caption{An illustration to the proof of Proposition~\ref{Prop:Boundary}. In all four sub-figures, the underlying graph is $G=CK(15)$, and we have $\alpha=2, \beta'=3,\gamma'=5,$ and $\delta=1$. The sub-figures present the SHP's $P_2,P_1,P_{-1},P_0$ in a clockwise order, starting with the left upper sub-figure. In each SHP, the corresponding diagonal $I_{\nu}$ is depicted by a punctured line. \textbf{To add `primes' to $\beta,\gamma$} }
\label{Fig:4-SHP's-2-holes}
\end{center}
\end{figure}

%
%

\medskip \noindent {\it Proof of~(2).} The union of the (closed) major sides of all the main diagonals $I_\nu$ $(-\delta \leq \nu \leq \alpha)$ is the boundary path $\langle 2m-1-\delta, \ldots,0, \ldots, m+\alpha \rangle$, of length $\delta+m+\alpha \leq 2m-4$. We call this path $J$, and split it into two vertex-disjoint parts: the upper part $J^{+}=\langle 2m-1-\delta, \ldots,0, \ldots, \beta' \rangle$ and the lower part $J^{-}=\langle \beta'+1, \ldots, m+\alpha \rangle$. Each edge of $P_\nu$ (except $h_1$ and $h_2$) has one vertex (call it the left one) in the interior of the minor side of $I_\nu$, and another vertex (call it the right one) in the closed major side of $I_\nu$. Note that in $P_{\nu}$, the right endpoint of each edge in the initial section lies in the upper part $J^+$ of $J$, whereas the right endpoint of each edge in the middle section and in the terminal section lies in the lower part $J^-$ of $J$.

It is clearly sufficient to show that for any $-\delta \leq \nu_1< \nu_2 \leq \alpha$ we have $P_{\nu_1} \cap P_{\nu_2} = \{h_1,h_2\}$. (Here $P_{\nu_1} \cap P_{\nu_2}$ denotes the set of edges (not the set of points) common to $P_{\nu_1}$ and $P_{\nu_2}$.) As for each $\nu$, the directions of the edges of $P_{\nu}$ (except for the edges $h_1,h_2$) are $2\nu-1,2\nu,2\nu+1,2\nu+2$, we may (potentially) have $P_{\nu_1} \cap P_{\nu_2} \neq \{h_1,h_2\}$ only if $\nu_2=\nu_1+1$ (as otherwise, there is no overlap between the directions of the edges that participate in these SHP's), and the joint edge must be of direction either $2\nu_2$ or $2\nu_2-1$. However, the edges of directions $2\nu$ and $2\nu-1$ in $P_{\nu}$ meet the major side of $I_{\nu}$ in $\{\nu,\nu+1,\ldots,\beta'\} \subset J^+$, while the edges of directions $2\nu$ and $2\nu-1$ in $P_{\nu-1}$ meet the major side of $I_{\nu-1}$ in $\{\gamma'+1,\gamma'+2,\ldots,m+\nu-1\} \subset J^-$. Thus, these SHP's cannot have an edge in common. This completes the proof of the claim, and therefore, also of Proposition~\ref{Prop:Boundary}.
\end{proof}

\section{The Diagonals of a Blocker -- Basic Restrictions}
\label{sec:Class_A}

In this section we prove basic restrictions that the diagonals of a blocker for SHP's in $CK(2m-1)$ must satisfy. These restrictions apply for blockers of both Class A and Class B. In the case of Class A, these constitute all restrictions stated in our main theorem, and they are sufficient, as we prove in Section~\ref{sec:converse}. (Hence, this section completes the proof of the `necessity' direction of our main theorem for blockers of Class~A.) For blockers of Class~B, there are additional restrictions which we will prove in the next section.

\medskip Recall that so far, we have shown that any blocker $B$ consists of $m$ edges of different directions, that these directions are those of a boundary path of length $m$ (w.l.o.g., $\langle 0,1,\ldots,m \rangle$), and that the boundary edges of $B$ form a boundary path $\langle \alpha,\alpha+1,\ldots,m-\delta \rangle$, which we call `the backbone of $B$', except for (possibly) a single missing edge which we denote $[\alpha+\beta,\alpha+\beta+1]$.


\begin{figure}[tb]
\begin{center}
\scalebox{1.0}{
\includegraphics{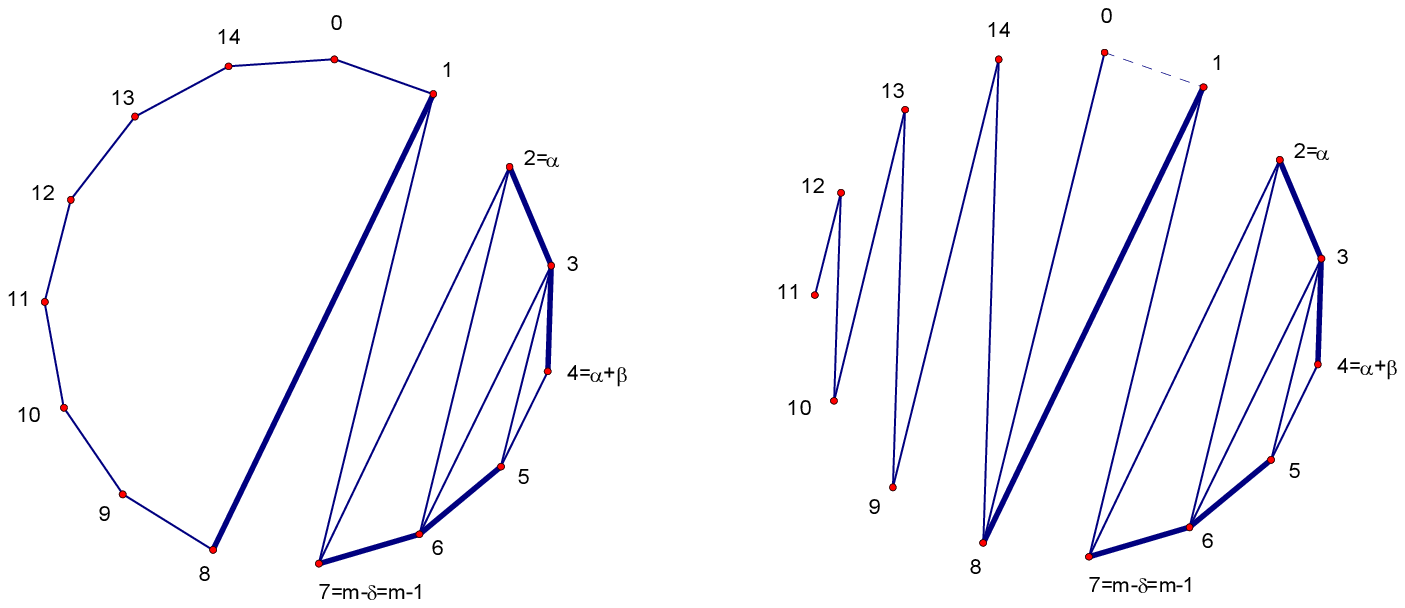}
} \caption{The left sub-figure is an illustration to the proof of Proposition~\ref{prop: beams}, and the right sub-figure is an illustration to the proof of Proposition~\ref{Prop:Aux-Class2}.} 
\label{fig:far_edge}
\end{center}
\end{figure}

Now we show that for any blocker $B$, each edge of $B$ that is not parallel to an edge in its backbone connects an internal vertex of the path $\langle \alpha,\alpha+1,\ldots,m-\delta \rangle$ with an internal vertex of the complementary path $\langle m-\delta,\ldots,2m-2,0,1,\ldots,\alpha \rangle$.

\begin{proposition}
\label{prop: beams}
Let $B$ be a blocker, and let $e \in B$ be an edge that is parallel to one of the $\alpha+\delta$ boundary edges $[0,1],\ldots,[\alpha-1,\alpha],[m-\delta,m-\delta+1],\ldots,[m-1,m]$. Put $A=\langle \alpha,\alpha+1,\ldots,m-\delta \rangle$ and $\bar{A}=\langle m-\delta,\ldots,2m-2,0,1,\ldots,\alpha \rangle$. Then $e$ connects an internal vertex of $A$ to an internal vertex of $\bar{A}$.
\end{proposition}

\begin{proof}
Assume, on the contrary, that the assertion fails, and thus, $e$ connects either two vertices of $A$ or two vertices of $\bar{A}$. (Note that $A$ and $\bar{A}$ meet at $\alpha$ and at $m-\delta$.) We consider two cases:

\mn \emph{Case~1: $e$ connects two vertices of $A$.} In this case, $e$ is parallel either to an edge of $A$, contrary to our initial assumption on $e$, or to a diagonal of order $2$ of $A$, which means that the direction of $e$ is $2i$, for some $\alpha < i< m-\delta$. However, $2i$ is not a direction of an edge of $B$. (Recall that the directions of edges of $B$ are $0,1,3,5,\ldots,2m-3 \mod (2m-1)$.) This contradicts the assumption $e \in B$.

\mn \emph{Case~2: $e$ connects two vertices of $\bar{A}$.} In this case we construct an SHP that misses $B$, as follows.

The endpoints of $e$ divide the boundary circuit of $P$ into two complementary closed arcs, $C$ and $\bar{C}$ (each containing the endpoints of $e$). Suppose $C$ includes the boundary arc $\langle \alpha, \ldots, m-\delta \rangle$. As $B$ contains an edge in the direction of $e$, at least one of the directions adjacent (in $\mathbb{Z}_{2m-1}$) to the direction of $e$ is not a direction of an edge of $B$. (Recall that $B$ contains only two adjacent directions, 0 and 1.) Call such a direction $x$. There is a zig-zag path $Z$ that uses the direction of $e$ and direction $x$ alternately, misses $e$, and covers all internal vertices of $C$ and one endpoint of $C$. The union of $Z$ and the boundary path of $\bar{C}$ is an SHP that avoids $B$ altogether.
(For example, in the case $G=CK(15)$, with $\alpha=2,\delta=1,$ and $e=[1,8]$, we have $C=\langle 1, \ldots, 8 \rangle$ and direction $x=8$ (which is adjacent to the direction $9$ of $e$) does not appear in $B$. Hence, we can define $Z=\langle 4,5,3,6,2,7,1 \rangle$ and obtain the SHP $\langle 4,5,3,6,2,7,1,0,14,13,12,11,10,9,8 \rangle$ that misses $B$, see the left part of Figure~\ref{fig:far_edge}.)
\end{proof}

We have shown that for any blocker $B$, each edge of $B$ that is not parallel to an edge in its backbone connects an internal vertex of the path $\langle \alpha,\alpha+1,\ldots,m-\delta \rangle$ with an internal vertex of the complementary path. We call such an edge a \emph{beam} of the blocker that \emph{emanates} from an internal point of $\langle \alpha,\alpha+1,\ldots,m-\delta \rangle$, and characterize all the possible sets of beams of a blocker. The first part of this characterization (Proposition~\ref{prop:go-far}) holds for blockers of both classes A and B. In this part we show that any such two beams do not intersect and in fact, turn away from each other. (Namely, the distance between the points they emanate from, is smaller than the distance between their other endpoints.) In the next section we present the second part of the characterization, which holds for blockers of Class~B.

\begin{proposition}
\label{prop:go-far}
Let $B$ be a blocker, and let $[i,j],[k,l] \in B$ be such that $\alpha < i < k < m - \delta$, $j,l \not \in \{ \alpha, \alpha+1, \ldots , m-\delta  \}$.
Then:
\begin{enumerate}
\item The beams $[i,j]$ and $[k,l]$ do not cross, i.e., the points $i,k,l,j$ appear in this order on the boundary of $\mathrm{conv}(V(G))$.
\item These two beams turn away from each other, i.e., $(k-i)$ is smaller than the distance between $l$ and $j$.
\end{enumerate}
\end{proposition}

\begin{proof}
Assume, on the contrary, that $[i,j]$ and $[k,l]$ are two edges of $B$ with $\alpha < i < k < m - \delta$, and $j,k  \in V \setminus \{ \alpha, \alpha+1, \ldots , m-\delta  \}$, and that one of the situations shown in Figure~\ref{fig:beams_go_far} occurs, i.e., that these two edges either cross or meet at an endpoint $(j=l)$, or approach each other ($j \equiv l + \zeta (\bmod (2m-1))$ for some $0< \zeta < k-i$).

To arrive at a contradiction, we produce an SHP $P$ that misses $B$ altogether. $P$ consists of two zig-zag paths $P_1$, $P_3$ connected by a boundary path $P_2$.

\emph{The path $P_1$.} The edge $[i,j]$ splits the boundary circuit of $\mathrm{conv}(V)$ into two closed paths $C^{+}$ and $C^{-}$, with $\alpha \in C^{+}$, $m- \delta \in C^{-}$. The direction of $[i,j]$ is $b \equiv i+j (\bmod (2m-1))$. One of the directions adjacent to $b$, (i.e., $b'=b+1$ or $b'=b-1$), is not a direction of edges in $B$. Let $Q$ be a zig-zag path consisting of edges of directions $b$ and $b'$ alternately that covers (exactly) all the internal vertices of $C^{-}$. One endpoint of $Q$ is either $i+1$ or $j-1$. If it happens to be $j-1$, put $P_1=Q$ (see Figure~\ref{fig:beams_two_zigzag}).
If it is $i+1$, then the (unique) edge of $Q$ that contains $i+1$ must be $[j-1,i+1]$, and therefore, the next edge of $Q$ is $[j-1,i+2]$, i.e., $b'=b+1$. In this case add to  $Q$ the edge $[i+1,j]$, and call the resulting path $P_1$. Thus $P_1$ is a zig-zag path with endpoint $j$ that covers all vertices of $C^{-}$ except $i$ (see Figure~\ref{fig:beams_two_zigzag2}).

\emph{The path $P_3$.} Let $[i,s]$ be the edge through $i$ parallel to $[k,l]$. Note that $s$ is an internal vertex of $C^{+}$. (If the line through $i$ parallel to $[k,l]$ does not meet another vertex, then $k+l \equiv 2i (\bmod (2m-1))$, but $2i$, being a positive even number, is not a direction of an edge of $B$. If $s$ is not in $C^{+}$, then $s$ lies between $i$ and $k$, and thus $[k,l]$ is either parallel to a boundary edge of $B$, or to a 2-diagonal of the path $\langle \alpha, \ldots, m-\delta \rangle$, and is not parallel to any edge of $B$.)

If the direction of $[k,l]$ is not 0 (see Figure~\ref{fig:beams_two_zigzag}), define $P_3$ to be the zig-zag path that covers all vertices of the boundary path $\langle s, \ldots, i \rangle$, and uses (alternately) edges parallel to $[i,s]$ (i.e., to $[k,l]$) and edges of direction $k+l+1$, with one endpoint at $s$. (Note that if $a \neq 0$ and $B$ contains an edge of direction $a$, then $B$ does not contain edges of direction $a+1$.)

If the direction of $[k,l]$ is 0 (see Figure~\ref{fig:beams_two_zigzag2}), let $Q$ be a zig-zag path with edges of direction 0 and -1 that covers all vertices of the boundary path $\langle s, \ldots, i \rangle$, with endpoint $i$. Add to $Q$ the edge $[s-1,i]$ (of direction -1) to form the desired path $P_3$ with endpoint $s-1$.

\emph{The path $P_2$.} $P_2$ is the boundary path that connects the endpoint $j-1$ or $j$ of $P_1$, to the endpoint $s$ or $s-1$ of $P_3$.

It is easy to see that the constructed SHP misses $B$, as asserted.

\end{proof}

\begin{figure}[tb]
\begin{center}
\scalebox{1.0}{
\includegraphics{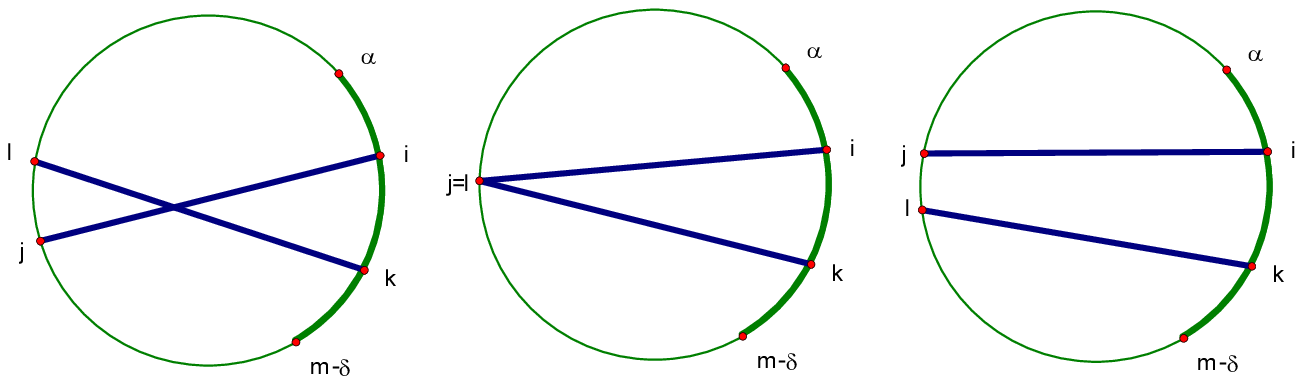}
} \caption{Illustration to the proof of Proposition~\ref{prop:go-far}.}
\label{fig:beams_go_far}
\end{center}
\end{figure}

\begin{figure}[tb]
\begin{center}
\scalebox{1.0}{
\includegraphics{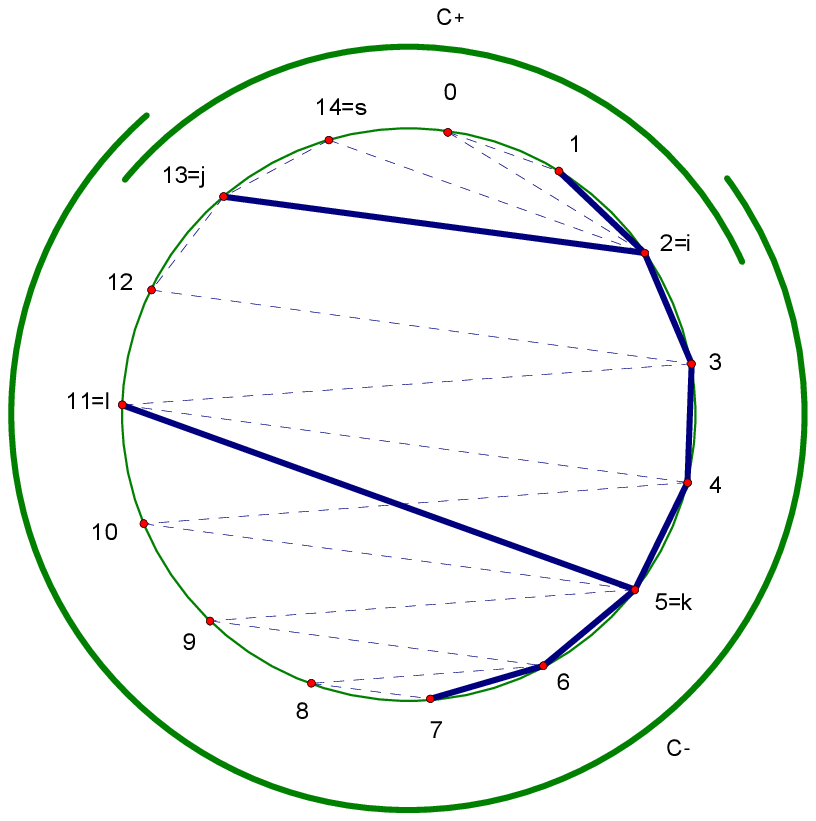}
} \caption{Illustration to the proof of Proposition~\ref{prop:go-far} in the case where the direction of $[k,l]$ is not $0$: $b=0$, $b'=b-1 \equiv 14$, $P_1= \langle 12,3,11,4, \ldots, 8,7 \rangle$, $P_3= \langle 14,2,0,1 \rangle$, $P_2= \langle 12,13,14 \rangle$.}
\label{fig:beams_two_zigzag}
\end{center}
\end{figure}

\begin{figure}[tb]
\begin{center}
\scalebox{1.0}{
\includegraphics{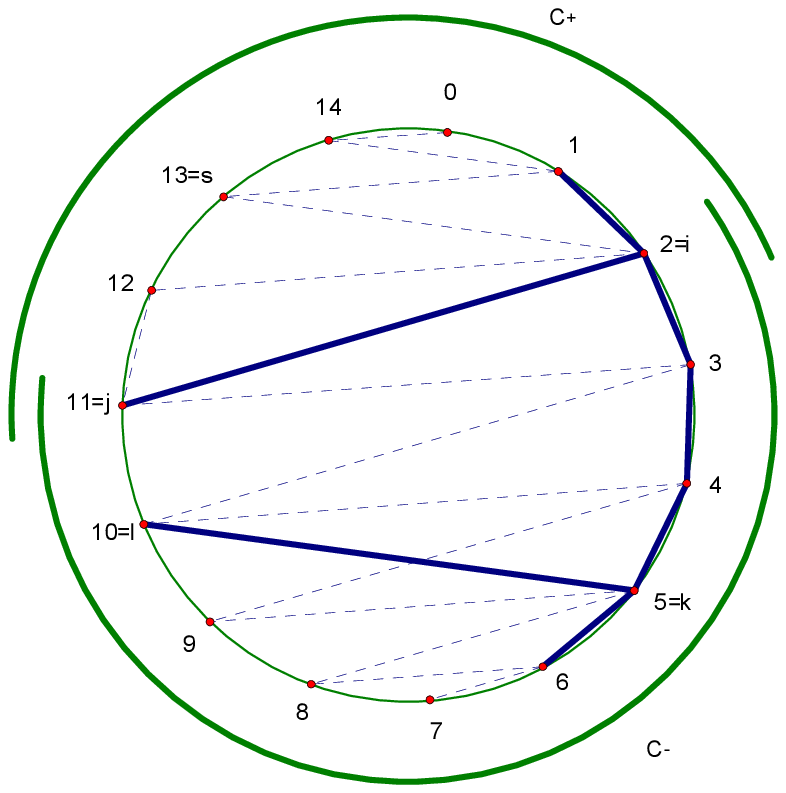}
} \caption{Illustration for the proof of Proposition~\ref{prop:go-far} in the case where the direction of $[k,l]$ is $0$: $b=13$, $b'=b+1 \equiv 14$, $P_1= \langle 11,3,10,4, \ldots, 6,7 \rangle$, $P_3= \langle 12,2,13,1,14,0 \rangle$, $P_2= \langle 11,12 \rangle$.}
\label{fig:beams_two_zigzag2}
\end{center}
\end{figure}

\section{The Diagonals of a Blocker of Class B - More Restrictions}
\label{sec:Class_B}

The restrictions described so far on the beams (the diagonals of a blocker) complete the characterization of the blockers of Class A; we proved that any blocker whose backbone is a single path is of the form described in Class A. In this section we prove that when this backbone contains a `hole' (as proved before - a single hole of length 1), the beams satisfy another requirement, as stated in the definition of Class B.

We start with determining the possible placements of the edge that `replaces' the missing edge in the backbone in a blocker of Class~B. We show that the edge that is parallel to the missing edge $[\alpha+\beta,\alpha+\beta+1]$ connects an internal vertex of the path $\langle \alpha,\alpha+1,\ldots,\alpha+\beta \rangle$ with an internal vertex of the path $\langle \alpha+\beta+1,\alpha+\beta+2,\ldots, m-\delta \rangle$.
\begin{proposition}\label{Prop:Aux-Class2}
Let $B$ be a blocker of Class~B, with backbone $\langle \alpha,\ldots,m-\delta \rangle$ and missing edge $[\alpha+\beta,\alpha+\beta+1]$. Then there exists a $\nu$, $1 \leq \nu  < \min(\beta,\gamma)$ (where $\beta$ is the length of the part of the backbone `above' the missing edge, and $\gamma=m-\delta-\alpha-\beta-1$ is the length of the part of the backbone `below' the missing edge), such that $[\alpha+\beta-\nu,\alpha+\beta+1+\nu] \in B$.
\end{proposition}

\begin{proof}
The blocker $B$ contains an edge in direction $2\alpha+2\beta+1$, and thus, there exists a $\nu$, $1 \leq \nu \leq m-2$, such that $[\alpha+\beta-\nu,\alpha+\beta+1+\nu] \in B$. Assume, on the contrary, that either $\nu \geq \beta$ or $\nu \geq \gamma$. Consider first the case $\nu \geq \beta$. We construct an SHP $H$ that misses $B$.

Divide the vertices of $CK(2m-1)$ into two complementary boundary paths:
\begin{itemize}
\item $A = \langle \alpha+\beta-\nu, \alpha+\beta-\nu+1 \ldots, \alpha+\beta+\nu \rangle$, and

\item $B = \langle \alpha+\beta-\nu-1, \alpha+\beta-\nu-2 \ldots, \alpha+\beta+\nu +1 \rangle$.
\end{itemize}
Cover the vertices of $A$ by a zig-zag path $P_1$ of length $2 \nu$, that starts at $\alpha+\beta-\nu$, uses edges of directions $2\alpha+2\beta,2\alpha+2\beta+1$ alternately, and terminates at $\alpha+\beta$. (For example, in the case $m=15,\alpha=2,\beta=2,\delta=1,\nu=3$, this is the path $\langle 1,7,2,6,3,5,4 \rangle$. See the right part of Figure~\ref{fig:far_edge}.)

Cover the vertices of $B$ by a zig-zag path $P_2$ of length $2m-2\nu-3$, that starts at $\alpha+\beta-\nu-1$, and uses edges in directions $2\alpha+2\beta,2\alpha+2\beta-1$ alternately. (In the aforementioned example (see Figure~\ref{fig:far_edge}), this is the path $\langle 0,8,14,9,13,10,12,11 \rangle$.)

Then, connect $P_1$ and $P_2$ by the boundary edge $[\alpha+\beta-\nu,\alpha+\beta-\nu-1]$, that does not belong to $B$ by the assumption on $\nu$. (In the aforementioned example, this is the edge $[1,0]$.)

It is easy to see that the resulting SHP misses $B$, contradicting the assumption that $B$ is a blocker.

\medskip The case $\nu \geq \gamma$ is handled similarly, as it is equivalent to the case $\nu \geq \beta$ under a reflection through the perpendicular bisector of the edge $[\alpha+\beta,\alpha+\beta+1]$.
\end{proof}

Let us denote the boundary paths $\langle \alpha, \ldots, \alpha+\beta \rangle$ and $\langle m-\gamma-\delta, \ldots, m-\delta \rangle$ by $B^{+}$ and $B^{-}$, respectively. Let us also denote by $e_i$ the edge of $B$ in direction $i$, and put $E^{+}=\{e_i | i=1,3,5,\ldots,2\alpha-1 \}$, $E^{-}=\{e_i | i= 2\alpha+2\beta+3, 2\alpha+2\beta+5,\ldots, 2m-2\delta-1\} = \{e_i: i=0,-2,-4,\ldots,2-2\delta\}$. By Proposition~\ref{prop: beams} we already know that each edge of $E^{+} \cup E^{-}$ meets exactly one vertex of the boundary path $\langle \alpha +1, \ldots, m- \delta - 1 \rangle$, and none of the vertices $\alpha, m-\delta$. Our claim is that each edge of $E^{+} (E^{-})$ meets an internal vertex of $B^{+} (B^{-})$.

The following proposition asserts the two requirements stated in parts (4)-(5) of the definition of Class B, namely, that the edges of $B$ that are parallel to $[0,1], \ldots, [\alpha-1,\alpha]$ lie `above' the missing boundary edge, and that the edges of $B$ that are parallel to $[m-\delta,m-\delta+1], \ldots, [m-1,m]$ lie `below' the missing boundary edge.

\begin{proposition}
\label{cl:test_paths}
Each edge of $E^{+} (E^{-})$ meets an internal vertex of $B^{+} (B^{-})$.
\end{proposition}

\begin{proof}
The proof of the proposition, which spans the rest of this section, consists of two steps:
\begin{itemize}
\item \textbf{Step~1:} We construct a set of $2(\alpha+\delta)$ SHPs (which we call ``test paths'') $F_i$, $1-2\delta \leq i \leq 2\alpha$, with the following properties:
\begin{enumerate}
\item Any edge common to $B$ and to some test path $F_i$ belongs to $E^{+}\cup E^{-}$.
\item Each edge of $E^{+}\cup E^{-}$ meets at most two $F_i$'s.
\end{enumerate}
Since $|E^{+}\cup E^{-}|=\alpha+\delta$, it follows that $B$ cannot block all the test paths $F_i$ unless each edge $e \in E^{+}\cup E^{-}$ meets exactly two $F_i$'s.

\item \textbf{Step~2:} We show, by an inductive argument, that the edges of $E^{+}\cup E^{-}$ meet the test paths $F_i$ in a particular order. As a consequence, we find that each edge of $E^{+}(E^{-})$ meets an internal vertex of $B^{+} (B^{-})$.
\end{itemize}

\medskip \noindent \textbf{Preliminaries}

\medskip Before starting the description of the test paths $F_i$, we need a few additional notations. Recall that the set of edges in direction $j$ is $D_j= \{ [x,y]| x,y \in \mathbb{Z}_{2m-1}, x \neq y, x+y=j \}$. This is a set of $m-1$ parallel edges, an almost perfect matching of $G$ that misses exactly one vertex (the vertex $x$ that satisfies $2x=j$).

Since directions are elements of $\mathbb{Z}_{2m-1}$, we may replace the indices that represent them by any sequence of $2m-1$ consecutive numbers. In our context, we find it convenient to use as indices the numbers
\[
2 \alpha +2 \beta -2m+3, \ldots, 2 \alpha +2 \beta , 2 \alpha +2 \beta +1.
\]
In the sequel we will be interested in edges that pass through either $\alpha+\beta$ or $\alpha+\beta+1$. We note that for $2 \alpha +2 \beta -2m+4 \leq i \leq 2 \alpha +2 \beta -1$ there are two different parallel edges in direction $i$, $d_i^{+}$ through $\alpha+\beta$ and $d_i^{-}$ through $\alpha+\beta+1$ (see Figure~\ref{fig:directions}). The exceptional values are:
\begin{itemize}
\item $2(\alpha +\beta)$. (No edge in this direction through $\alpha +\beta$.)
\item $2(\alpha +\beta +1)=2\alpha +2 \beta -2m+3$. (No edge in this direction through $\alpha +\beta +1$.)
\item $2(\alpha +\beta) +1$. (The direction of $[\alpha +\beta,\alpha +\beta+1]$.)
\end{itemize}

\begin{figure}[tb]
\begin{center}
\scalebox{1.0}{
\includegraphics{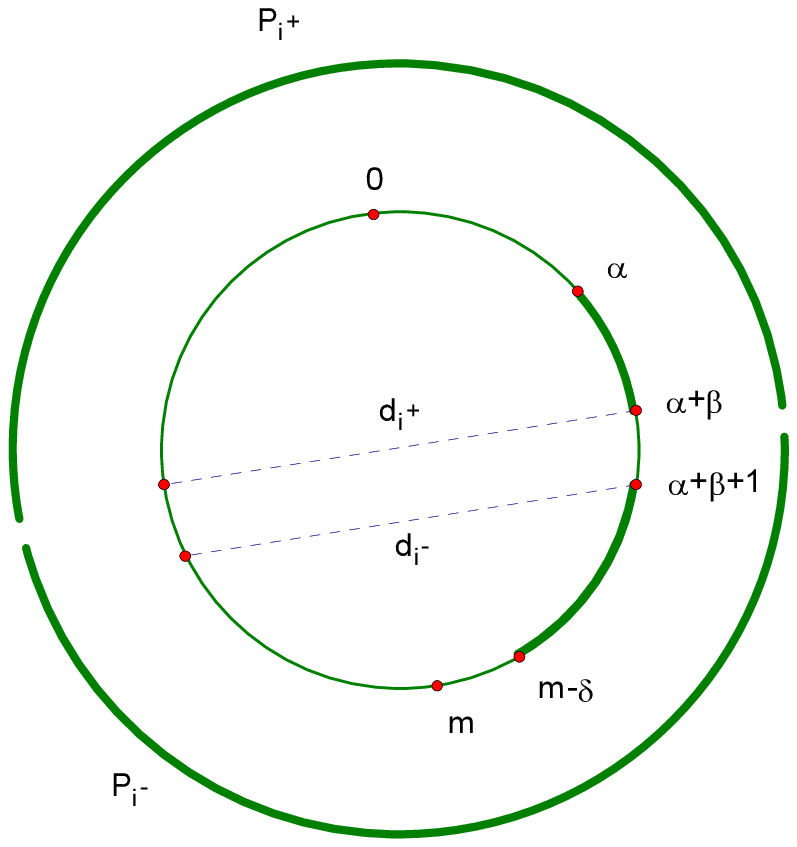}
} \caption{Illustration for the proof of Proposition~\ref{cl:test_paths}.}
\label{fig:directions}
\end{center}
\end{figure}

For each $i$, $2\alpha+2\beta-2m+4 \leq i \leq 2\alpha+2\beta-1$, let $P_i^{+}$ be the intersection of $P$ with the closed half-plane bounded by $\mathrm{aff}(d_i^{+})$ that misses $d_i^{-}$, and let $P_i^{-}$ be the intersection of $P$ with the closed half-plane bounded by $\mathrm{aff}(d_i^{-})$ that misses $d_i^{+}$, as shown in Figure~\ref{fig:directions}.

Note that as $i$ increases from $2\alpha +2 \beta -2m+4(=2-2\gamma-2\delta)$ to $2\alpha+2\beta-1$, the sequence of polygons $P_i^{-}$ increases $(P_i^{-} \subset P_{i+1}^{-})$, whereas the sequence $P_i^{+}$ decreases $(P_{i+1}^{+} \subset P_i^{+})$. We call this property ``the monotonicity of $P_i$''.

\medskip \noindent \textbf{The definition of the $F_i$'s}

\medskip We are ready to define the ``test paths'' $F_i$. The set $D_i$ splits into two subsets: $D_i^{+}=\{ d \in D_i | d \subset P_i^{+} \}$ and $D_i^{-}=\{ d \in D_i | d \subset P_i^{-} \}$. $D_i^{+}$ is a simple perfect matching in $P_i^{+}$, and $D_i^{-}$ is an almost simple perfect matching in $P_i^{-}$ (or vice versa).
By adding to $D_i^{+}$ all edges in direction $i-1$ that lie in $P_i^{+}$, we obtain a zig-zag path, which is an SHP of $P_i^{+}$, that ends with an edge in direction $i$ at $\alpha + \beta$. We call this path $F_i^+$ (see Figure~\ref{fig:F_i} at the end of the paper).
Similarly, by adding to $D_i^{-}$ all edges in direction $i+1$ that lie in $P_i^{-}$ we obtain a zig-zag path, an SHP $F_i^{-}$ of $P_i^{-}$, that starts at $\alpha+\beta+1$ with an edge in direction $i$. Connecting $F_i^{+}$ to $F_i^{-}$ by the edge $[\alpha+\beta,\alpha+\beta+1]$ we obtain the path $F_i$. That is, $F_i= F_i^+ \cup F_i^- \cup \{[\alpha+\beta,\alpha+\beta+1]\}$. An example of the paths $F_i$, $-2\delta+1 \leq i \leq 2\alpha$, is presented in Figure~\ref{fig:F_i} at the end of the paper.

\medskip \noindent \textbf{Properties of the $F_i$'s}

\medskip Now we show that the $F_i$'s satisfy the aforementioned Properties (1)-(2), thus accomplishing Step~1 of the proof. For this, we need a few additional notations.

Within the range $2\alpha+2\beta-2m+4 \leq i \leq 2\alpha+2\beta-1$, we distinguish two special directions $\iota^-$ and $\iota^+$, by requiring that $d^-_{\iota^-}$ connects the two endpoints of $B^-$ and $d^+_{\iota^+}$ connects the two endpoints of $B^+$. Thus
\begin{equation}\label{Eq:Iota}
\iota^- = \alpha+\beta+1+m-\delta = m-\gamma-\delta+m-\delta, \qquad \mbox{and} \qquad \iota^+= \alpha+\alpha+\beta=2\alpha+\beta.
\end{equation}
In order to place the direction $\iota^-$ at the range $[2\alpha+2\beta-2m+3,2\alpha+2\beta+1]$, we subtract $2m-1$, and obtain $\iota^- = 2m-\gamma-2\delta-2m+1= 1-\gamma-2\delta$. And indeed, $2\alpha-2\beta-2m+4 < 1-\gamma - 2\delta \leq 2\alpha+\beta \leq 2\alpha+2\beta-1$.

For our proof, we need the paths $F_i$ only for $1-2\delta \leq i \leq 2\alpha $. (Note that $2\alpha +2 \beta -2m+4 < 1-2\delta \leq 2\alpha +2 \beta-1$,  since $\alpha+\beta+\gamma+\delta = m-1$.) The following claim asserts that in this range of $i$'s, we always have $B^+ \subset P_i^+$ and $B^- \subset P_i^-$.
\begin{claim}
\label{prop:B_in_P_i}
For all $1-2\delta \leq i \leq 2\alpha$, $B^{-} \subset P_i^{-}$ and $B^{+} \subset P_i^{+}$.
\end{claim}
\begin{proof}[Proof of Claim~\ref{prop:B_in_P_i}]
By the monotonicity of the $P_i$'s, it is sufficient to prove that the range between $2\alpha$ and $1-2\delta$ is part of the range between $\iota^-$ and $\iota^+$. This indeed holds, as by~\eqref{Eq:Iota} we have $\iota^- =1-\gamma-2\delta<1-2\delta$ and $\iota^+ = 2\alpha+\beta>2\alpha$.
\end{proof}
%

The following claim shows that the $F_i$'s satisfy Property~(1) stated at the beginning of the proof of Proposition~\ref{cl:test_paths}.
\begin{claim}
\label{prop:test_paths_meet_B}
Our test paths $F_i$, $-2\delta+1 \leq i \leq 2\alpha$ do not meet any edge of $B^{+} \cup \{e_{2\alpha+2\beta+1}\} \cup B^{-}$. I.e., they meet $B$ in $E^{+} \cup E^{-}$ only.
\end{claim}

\begin{proof}[Proof of Claim~\ref{prop:test_paths_meet_B}]
First note that the boundary edge $[\alpha + \beta, \alpha + \beta+1]$ belongs to all $F_i$, but is not in $B$. Thus no $F_i$ uses $e_{2\alpha + 2\beta+1}$. Beyond that, each $F_i$ uses only edges in directions $i-1$,$i$ and $i+1$. Thus the directions of the edges in $\bigcup \{ F_i | 1-2\delta \leq i \leq 2 \alpha\}$ (excluding $[\alpha + \beta, \alpha + \beta+1]$) are $$-2\delta, \ldots ,0, \ldots,  2\alpha+1,$$ as demonstrated in Figure~\ref{fig:F_i}. The directions of the edges of $B_{+}=\langle \alpha, \ldots \alpha+\beta \rangle$ are the odd positive numbers $$2\alpha+1, \ldots, 2\alpha+2\beta-1.$$ The directions of the edges of $B_{-}=\langle m-\gamma-\delta, \ldots m-\delta \rangle$ are $2m-2\gamma-2\delta+1, \ldots, 2m-2\delta-1$ $(\bmod (2m-1))$. In order to embed them in the interval $[2\alpha+2\beta -2m+3,2\alpha+2\beta+1 ]$ we subtract $2m-1$ to get the even negative numbers $$-2\gamma-2\delta+2, \ldots, -2\delta.$$

Thus the only directions common to edges of $F_i$ $(1-2\delta \leq i \leq 2\alpha)$ and of $B^{-} \cup B^{+}$ are $-2\delta$ (the direction of $[m-\delta-1, m-\delta]$ in $B^{-}$, and of some edges in $F_{1-2\delta}$) and $2\alpha+1$ (the direction of $[\alpha,\alpha+1]$ in $B^{+}$, and of some edges in $F_{2\alpha}$).

But the edges of $F_{1-2\delta}$ in direction $-2\delta$ lie in $P_{1-2\delta}^{+}$, whereas $B^{-}$ is part of $P_{1-2\delta}^{-}$. Similarly, the edges of $F_{2\alpha}$ in direction $2\alpha+1$ lie in $P_{2\alpha}^{-}$, whereas $B^{+}$ is part of $P_{2\alpha}^{+}$. This completes the proof of the claim.
\end{proof} 

The following claim shows that the $F_i$'s satisfy Property~(2) stated at the beginning of the proof of Proposition~\ref{cl:test_paths}.
\begin{claim}
\label{prop:each_edge_meets_at_most_2}
Each edge of $E^{+}\cup E^{-}$ meets at most two $F_i$'s.
\end{claim}

\begin{proof}[Proof of Claim~\ref{prop:each_edge_meets_at_most_2}]
By Claim~\ref{prop:test_paths_meet_B}, the only edges of $B$ that may meet our test-paths $F_i$ $(1-2\delta \leq i \leq 2\alpha)$ are: $e_1,e_3,\ldots, e_{2\alpha-1}$ (in $E^{+}$) and $e_0,e_{-2},e_{-4}, \ldots, e_{2-2\delta}$ (in $E^{-}$). (Recall that $e_{2\mu-1}$ denotes the edge of $E^{+} \subset B$ parallel to $[\mu-1,\mu]$ ($\mu=1,\ldots,\alpha$), and $e_{-2\nu}$ is the edge of $E^{-} $ parallel to $[m-\nu-1,m-\nu]$ ($\nu=0,\ldots,\delta-1$) ).

Note that the only test paths $F_i$ that include edges in direction $2\mu-1$ $(1 \leq \mu \leq \alpha)$ are $F_{2\mu-2},F_{2\mu-1}$ and $F_{2\mu}$. Moreover, the edges of $F_{2\mu-2}$ in direction $2\mu-1$ lie in $P_{2\mu-2}^{-}$, whereas the edges of $F_{2\mu}$ in direction $2\mu-1$ lie in $P_{2\mu}^{+}$, and $P_{2\mu-2}^{-} \cap P_{2\mu}^{+} \subset P_{2\mu}^{-} \cap P_{2\mu}^{+} = \emptyset$. This means that $e_{2\mu-1}$ can meet at most two test-paths $F_i$. A similar argument shows that each edge $e_{-2\nu} \in E^{-}$ $(\nu=0,1,\ldots,\delta-1)$ can meet at most two  test-paths $F_i$.
\end{proof}

\medskip \noindent \textbf{Step~2 of the proof}

\medskip

So far we accomplished Step~1 of the proof. We have constructed $2(\alpha+\delta)$ test-paths $F_i$ $(-2\delta+1 \leq i \leq 2\alpha)$, to be met by $\alpha+\delta$ edges of $E^{-} \cup E^{+}$, where each edge meets at most two $F_i$'s. Thus each edge of $E^{-} \cup E^{+}$ must meet exactly two test-paths. \begin{claim}\label{Claim:only_possible}
The only possible way in which the above can hold is if for all $\mu=1,\ldots,\alpha$, the edge $e_{2\mu-1}$ meets $F_{2\mu}$ and $F_{2\mu-1}$, and for all  $\nu=0,\ldots,\delta$, the edge $e_{-2\nu}$ meets $F_{-2\nu}$ and $F_{1-2\nu}$.
\end{claim}

\begin{proof}[Proof of Claim~\ref{Claim:only_possible}]
By induction. The base is the cases $(\alpha,\delta)=(1,0)$ and $(\alpha,\delta)=(0,1)$. In the case $(\alpha,\delta)=(1,0)$, there are only two paths: $F_1$ and $F_2$. The directions in $F_2$ are 3 (which is not a direction of any edge of $E^+$), 2 (which is not a direction of any edge in the blocker) and 1. Hence, $e_1$ must meet $F_2$. Therefore, it emanates from $B^+$ and must meet $F_1$ as well. The case $(\alpha,\delta)=(0,1)$ is proved similarly. 
The induction steps are $(\alpha-1,\delta) \rightarrow (\alpha,\delta) $ if $\alpha>0 $, and $(\alpha,\delta-1) \rightarrow (\alpha,\delta) $ if $\delta>0 $). If $\alpha>0 $ then $F_{2\alpha}$ must include $e_{2\alpha-1}$ (as $e_{2\alpha+1}$ is out of the range and the direction $2\alpha$ does not appear in the blocker). Hence, $e_{2\alpha-1}$ emanates from $B^+$. It follows that $F_{2\alpha-2 }$ does not include $e_{2\alpha-1}$ (since the edges in direction $2\alpha-1$ in $F_{2\alpha-2}$ emanate from $B^-$), and therefore $e_{2\alpha-1}$ meets just $F_{2\alpha}$ and $F_{2\alpha-1}$. Put the edge $e_{2\alpha-1}$ and the paths $F_{2\alpha}, F_{2\alpha-1}$ aside, and you are left with the case $(\alpha-1,\delta)$.

A similar argument shows that if $\delta>0$, then $e_{-2\delta}$ must meet $F_{-2\delta}$ and  $F_{2-2\delta}$.
\end{proof}

Now comes our final conclusion: For $\mu=1,\ldots,\alpha$, $e_{2\mu-1}$ is an edge of $F_{2\mu}$ in direction $2\mu-1$ and therefore lies in $P_{2\mu}^{+}$ and does not meet $B^{-} \subset P_{2\mu}^{-}$. We already know that $e_{2\mu-1}$ meets the boundary path $\langle \alpha, \ldots, m-\delta \rangle$ in a (unique) internal vertex, thus we find that $e_{2\mu-1}$ meets $B^{+}=\langle \alpha, \ldots, \alpha+\beta \rangle$, but not in $\alpha$. It does not contain $\alpha+\beta$, since the only edges of $F_{2\mu}$ at $\alpha+\beta$ are $[2\mu-\alpha-\beta+2m-1,\alpha+\beta]$ and $[\alpha+\beta,\alpha+\beta+1]$. In other words, $e_{2\mu-1}$ meets an internal vertex of $B^{+}=\langle \alpha, \ldots, \alpha+\beta \rangle$.
In a symmetric way, we can show that $e_{-2\mu}$ $( 0 \leq \mu \leq \delta-1)$ passes through an internal vertex of $B^{-}=\langle \alpha+\beta+1, \ldots, m-\delta \rangle$. This completes the proof of Proposition~\ref{cl:test_paths}.
\end{proof}

Proposition~\ref{cl:test_paths} completes the restrictions on the beams of a blocker with a broken backbone, and together with the restrictions that were described in the previous two sections, this completes one direction of the proof -- each blocker satisfies the conditions of Class A or Class B that we defined in the introduction.

\section{Proof of the Converse Direction}
\label{sec:converse}

In this section we prove the converse direction of the main theorem: any subset of edges of $G$ of one of the forms of Class A and Class B, as described in the introduction, has an edge in common with any SHP. First we prove this for the elements of Class A, and then we expand the proof to elements of Class B.

\subsection{The elements of Class A are blockers}

We start with a description of the elements of Class A, that is somewhat different from the description given in the introduction, and is more convenient for the arguments that are needed in this section.

\textbf{An alternative description of Class A:}
We regard the boundary circuit of $CK(2m-1)$ as composed of two vertex disjoint paths (see Figure~\ref{fig:opp_dir_full_spine}): The (short) $a$-path $\langle a_0, \ldots, a_{m-\alpha-\delta} \rangle$, and the (long) $b$-path $\langle b_0, \ldots, b_{m-3+\alpha+\delta} \rangle$, with $b_0$ adjacent to $a_0$ and $b_{m-3+\alpha+\delta}$ adjacent to $a_{m-\alpha-\delta}$. Here, as above, $\alpha \geq 0$, $\delta \geq 0$, and $\alpha+\delta \leq m-2$. (The $a$-path is longer than the $b$-path only when $\alpha+\delta \leq 1$.) The set $B$ is the union of the boundary $a$-path ($m-\alpha-\delta$ edges) and $\alpha+\delta$ diagonals $[a_{i_{\nu}},b_{j_{\nu}}]$, $\nu=1,\ldots,\alpha+\delta$, where $$1 \leq i_1 \leq i_2 \leq \ldots \leq i_{\alpha+\delta} \leq m-\alpha-\delta-1,$$ and
\begin{equation}\label{Eq:Kochavit}
j_{\nu} = \left\lbrace
  \begin{array}{c l}
    i_{\nu}+2\nu-2, & 1 \leq \nu \leq \alpha\\
    i_{\nu}+2\nu-3, & \alpha+1 \leq \nu \leq \alpha+\delta.
  \end{array}
\right.
\end{equation}

\begin{figure}[tb]
\begin{center}
\scalebox{1.2}{
\includegraphics{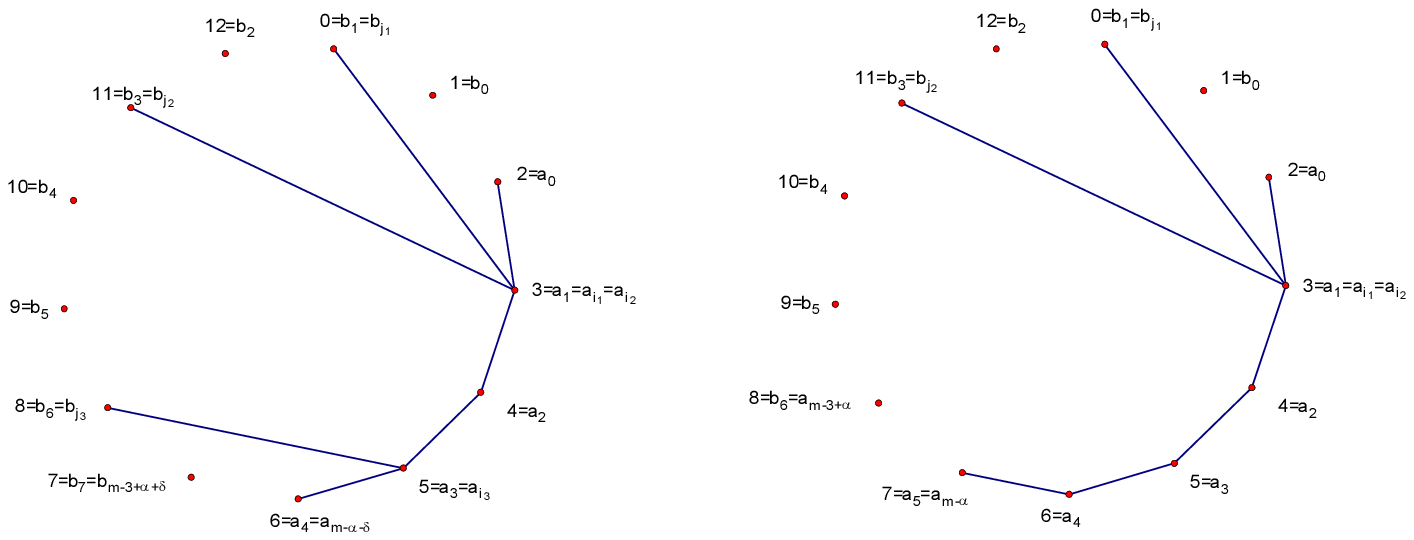}
} \caption{Illustration for the alternative description of Class A's elements. In the left figure, $m=7, \alpha=2$ and $\delta=1$, and in the right figure,
$m=7,\alpha=2$ and $\delta=0$.}
\label{fig:opp_dir_full_spine}
\end{center}
\end{figure}


All these diagonals connect an internal vertex of the $a$-path with a vertex of the $b$-path. As we increase $\nu$ by one, the increment in $j_{\nu}$ is usually equal to the increment in $i_{\nu}$ plus two. Only in passing from $\nu=\alpha$ to $\nu=\alpha+1$ the increment in $j_{\nu}$ equals the increment in $i_{\nu}$ plus one. (Note that $j_1=i_1$ (provided $\alpha>0$), and $m-3+\alpha+\delta-j_{\alpha+\delta} = m-\alpha-\delta-i_{\alpha+\delta}$ (provided $\delta>0$).) This description is valid also in the extreme cases, when $\alpha=0$ or $\delta=0$. The only cases where the $b$-vertex of a diagonal is not an internal vertex of the $b$-path are $(\delta=0,i_{\alpha}=m-\alpha-1,j_{\alpha}=m-3+\alpha)$ and $(\alpha=0,i_1=1,j_1=0)$.

\begin{theorem}
Any subset of edges $B \subset E(G)$ of the type described in Class A is a blocker for SHP's in $CK(2m-1)$.
\label{thm:opp_direc_A}
\end{theorem}

In the proof of the theorem we use the following observation:
\begin{observation}
\label{obs:connect_a_path}
Let $B$ be a subset of $E(CK(2m-1))$ that conforms to the description of Class A, and let $F$ be an SHP in $CK(2m-1)$. If an edge of $F$ connects two vertices of the $a$-path, then $F \cap B \neq \emptyset$.
\end{observation}

\begin{proof}[Proof of Observation~\ref{obs:connect_a_path}]
Suppose $[a_i,a_j] \in F$ with $0 \leq i < j \leq m-\alpha -\delta$. Choose such an edge with $j-i$ as small as possible. If $j-i=1$, then $[a_i,a_j] \in F \cap B$. If $j-i>1$, then we cannot continue $F$ into $\mathrm{conv}\{a_{\nu} | i \leq \nu \leq j \}$.
\end{proof} 




\begin{proof}[Proof of Theorem~\ref{thm:opp_direc_A}]
First we consider the extreme case $\alpha=\delta=0$. In this case, $B$ is the boundary path $\langle 0 , \ldots, m \rangle$. If $P$ is an SHP that avoids $B$, then by Observation~\ref{obs:connect_a_path} no edge of $P$ connects two vertices of $\langle 0 , \ldots, m \rangle$. In other words, each edge of $P$ uses at least one of the $m-2$ remaining vertices of $CK(2m-1)$. Thus $P$ has at most $2(m-2)$ edges, which is impossible.

\medskip Now we assume $\alpha+\delta>0$ and show that $B$ is a blocker in this case as well.

Assume, on the contrary, that there exists a Hamiltonian path $F$ that misses $B$. Then by Observation~\ref{obs:connect_a_path}, each internal vertex $a_i$ of the $a$-path, $0 < i < m-\alpha-\delta$, must be an internal vertex of the SHP $F$ as well, and its two neighbors in $F$ must be the two adjacent vertices $b_{k(i)}$ and $b_{k(i)+1}$ of the $b$-path. Moreover, the function $i \mapsto k(i)$ must be strictly monotone. (Recall that if $z$ is an extremal vertex of an SHP $F$ then the only edge of $F$ which emanates from $z$ is a boundary edge; on the other hand, if $a_i$ is an internal vertex of the $a$-path, then both boundary edges that are incident to $a_i$ belong to $B$, and thus cannot belong to $F$.)

\begin{claim}\label{Cl:Aux_new1}
For any $1 \leq \nu \leq \alpha+\delta$, we have $k(i_{\nu})>j_{\nu}$.
\end{claim}

\begin{proof}[Proof of Claim~\ref{Cl:Aux_new1}]
By induction on $\nu$.

\medskip \noindent \textbf{Induction basis.} Assume $\nu=1$. We consider two cases:
\begin{enumerate}
\item $\alpha>0.$ In this case, $j_1=i_1$. Note that $k(i_1) \geq i_1$ since $k(1) \geq 1$ and the function $i \mapsto k(i)$ is strictly increasing. We cannot have $k(i_1)=i_1$, as $j_1=i_1$ and so $[a_{i_1},b_{i_1}] \in B$. Hence, $k(i_1)>i_1=j_1$, as asserted.

\item $\alpha=0.$ In this case, $j_1=i_1-1$. Since $k(i_1) \geq i_1$, as in the first case, we have $k(i_1) \geq i_1>i_1-1 \geq j_1$, as asserted.
\end{enumerate}

\medskip \noindent \textbf{Induction step.} We assume $k(i_{\nu})>j_{\nu}$, and want to prove that $k(i_{\nu+1})>j_{\nu+1}$. We consider three cases:
\begin{enumerate}
\item $\nu \leq \alpha-1$ and so $\nu+1 \leq \alpha$.
\item $\nu=\alpha$ and so $\nu+1=\alpha+1$.
\item $\nu \geq \alpha+1$.
\end{enumerate}
Let $i_{\nu+1}=i_{\nu}+x$. Note that in Cases 1 and 3 we have $j_{\nu+1}=j_{\nu}+x+2$, and in Case 2 we have $j_{\nu+1}=j_{\nu}+x+1$.
\begin{itemize}
\item Cases 1,3: As $k(i_{\nu}) \geq j_{\nu}+1$ by assumption, we get $k(i_{\nu+1})=k(i_{\nu}+x) \geq j_{\nu}+1+x=j_{\nu+1}-1$, by the monotonicity of $k$. Since $[a_{i_{\nu+1}},b_{j_{\nu+1}}] \in B$, we can have neither $k(i_{\nu+1})=j_{\nu+1}-1$ nor $k(i_{\nu+1})=j_{\nu+1}$. Hence, $k(i_{\nu+1}) \geq j_{\nu+1}+1$, as asserted.

\item Case 2: By the same monotonicity argument as in the previous case, we have $k(i_{\nu+1}) \geq j_{\nu}+1+x=j_{\nu+1}$. As we cannot have $k(i_{\nu+1})=j_{\nu+1}$, this implies $k(i_{\nu+1}) \geq j_{\nu+1}+1$, as asserted.
\end{itemize}
This completes the proof of Claim~\ref{Cl:Aux_new1}.
\end{proof}

Now we return to the proof of the theorem. Recall that we assumed, to reach a contradiction, that the Hamiltonian path $F$ misses $B$. We consider two cases:

\medskip \noindent \textbf{Case 1: $\delta \neq 0$.} In this case, by~\eqref{Eq:Kochavit}, $j_{\alpha+\delta}=i_{\alpha+\delta}+2\alpha+2\delta-3$. Hence, the number of vertices of the $a$-path that follow $a_{i_{\alpha+\delta}}$ is $t=m-\alpha-\delta-i_{\alpha+\delta}$, and the number of vertices of the $b$-path that follow $b_{j_{\alpha+\delta}}$ is
\[
m-3+\alpha+\delta-j_{\alpha+\delta} = m-3+\alpha+\delta-i_{\alpha+\delta} -2\alpha-2\delta+3=m-\alpha-\delta-i_{\alpha+\delta} = t.
\]
As by Claim~\ref{Cl:Aux_new1} we have $k(i_{\alpha+\delta})>j_{\alpha+\delta}$, this implies that the part of $F$ that follows the edge $[a_{i_{\alpha+\delta}},b_{k(i_{\alpha+\delta})+1}] \in F$ (excluding the vertex $b_{k(i_{\alpha+\delta})+1}$) contains $t$ vertices of the $a$-path and at most $t-2$ vertices of the $b$-path. This is impossible since no edge of $F$ connects two vertices of the $a$-path, a contradiction.

\medskip \noindent \textbf{Case 2: $\delta =0$ and $\alpha \neq 0$.} In this case, by~\eqref{Eq:Kochavit}, $j_{\alpha+\delta}=j_{\alpha}=i_{\alpha}+2\alpha-2$. Hence, the number of vertices of the $a$-path which follow $a_{i_{\alpha}}$ is $t'=m-\alpha-i_{\alpha}$, and the number of vertices of the $b$-path which follow $b_{j_{\alpha}}$ is
\[
m-3+\alpha-j_{\alpha} = m-3+\alpha-i_{\alpha} -2\alpha+2=m-\alpha-i_{\alpha}-1 = t'-1.
\]
As by Claim~\ref{Cl:Aux_new1} we have $k(i_{\alpha})>j_{\alpha}$, this implies that the part of $F$ which follows the edge $[a_{i_{\alpha}},b_{k(i_{\alpha})+1}] \in F$ contains $t'$ vertices of the $a$-path and at most $t'-3$ vertices of the $b$-path. As in the previous case, this is impossible, since no edge of $F$ connects two vertices of the $a$-path, a contradiction.

This completes the proof of the theorem.
\end{proof}

\begin{remark}
We note that the above proof that the elements of Class A are blockers can replace the proof of the sufficiency direction in the case of even $n$ presented in~\cite{KP16}.
\end{remark}

\subsection{The elements of Class B are blockers}

In this subsection we prove the converse direction for the elements of Class B. Like in the case of Class A, we start with an alternative description of the elements of Class B.

\textbf{An alternative description of Class B:}
In the elements of Class B, the boundary path is broken. Each such element can be obtained from some blocker with an unbroken boundary path (i.e., a blocker of Class A), with parameters $m,\alpha,\delta$ (where $\alpha \geq 0,$ $\delta \geq 0,$ and $\alpha+\delta \leq m-2$), in the following way. We introduce three additional parameters $\beta,\gamma,\epsilon$, with $\beta \geq 2, \gamma \geq 2, \beta+1+\gamma = m-\alpha-\delta,$ and $1 \leq \epsilon < \min(\beta,\gamma)$. (Note that this is possible only when $m-\alpha-\delta \geq 5$). We remove from the boundary path $\langle a_0,\ldots,a_{m-\alpha-\delta} \rangle$ of $B$ the edge $h=[a_{\beta},a_{\beta+1}]$, thus splitting it into an ``upper'' $a$-path $\langle a_0,\ldots,a_{\beta} \rangle$ of length $\beta$ and a ``lower'' $a$-path $\langle a_{\beta+1},\ldots,a_{m-\alpha-\delta} \rangle$ of length $\gamma$. Then we replace the removed edge $[a_{\beta},a_{\beta+1}]$ by the parallel edge $[a_{\beta-\epsilon},a_{\beta+1+\epsilon}]$. This new edge connects an internal vertex of the upper $a$-path with an internal vertex of the lower $a$-path, as shown in Figure~\ref{fig:opp_dir_broken_spine}. We also make sure that the first $\alpha$ beams (i.e., the beams in directions $1,3,\ldots,2\alpha-1$) emanate from internal vertices of the upper $a$-path, and the last $\delta$ beams emanate from internal vertices of the lower $a$-path.

Let us call this modified set
\[
B(\epsilon) = B \setminus [a_{\beta},a_{\beta+1}] \cup [a_{\beta-\epsilon},a_{\beta+1+\epsilon}].
\]
(We may call the unmodified blocker $B=B(0)$.)

\begin{figure}[tb]
\begin{center}
\scalebox{0.9}{
\includegraphics{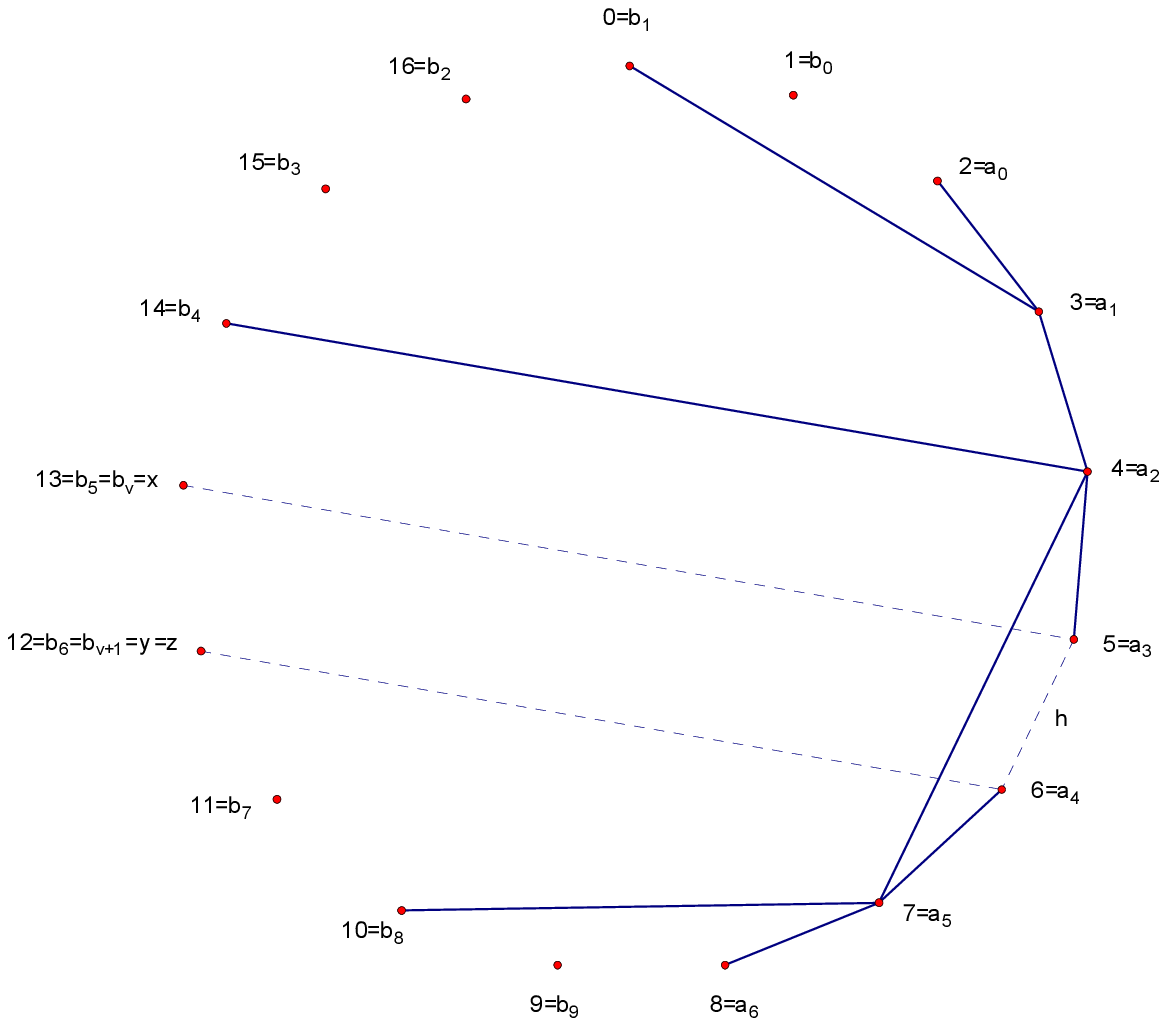}
} \caption{An illustration for notations in the proof of Theorem~\ref{thm:opp_direc_B}. The blocker $B(\epsilon)$ is depicted in regular lines, where the parameters are $m=9,\alpha=2,\delta=1$, $\beta=3,\gamma=2,$ and $\epsilon=1$. A part of the Hamiltonian path $P$ is depicted in punctured lines. $G^+$ is spanned by the vertices $\{ 5,4,3,2,1,0,16,15,14,13 \}$, and $G^-$ is spanned by the vertices $\{ 6,7,8,9,10,11 \}$.}
\label{fig:opp_dir_broken_spine}
\end{center}
\end{figure}

\begin{theorem}
Any subset of edges $B \subset E(G)$ of the type described in Class B is a blocker for SHP's in $CK(2m-1)$.
\label{thm:opp_direc_B}
\end{theorem}

\begin{proof}
By the alternative description, it is sufficient to show that for each $B$ that belongs to Class A and for each possible $\beta,\gamma,\epsilon$, the set $B(\epsilon)$ is a blocker for SHP's in $CK(2m-1)$. Suppose it is not, and assume $P$ is an SHP that avoids $B(\epsilon)$. $P$ must use the edge $h=[a_{\beta},a_{\beta+1}]$, since otherwise it would avoid the blocker $B$ as well. $h$ may be either an internal edge or a leaf edge of $P$. We consider these two cases separately.

\medskip \noindent \textbf{Case I: $h$ is a leaf edge of $P$.} In this case we use induction on $\epsilon$. Note that if $P$ misses the edge $[a_{\beta-\iota},a_{\beta+1+\iota}]$ for some $\iota$, $1 \leq \iota <\epsilon$, then $P$ avoids $B(\iota)$. Thus we may assume that $P$ uses $[a_{\beta-\iota},a_{\beta+1+\iota}]$ for $\iota=0,1,\ldots,\epsilon-1$. Furthermore, note that all the boundary edges incident with $a_i$, $\beta-\epsilon \leq i \leq \beta+1+\epsilon$, belong to $B(\epsilon)$. Thus $P$ has no choice, but to zigzag its way through these edges, emerging from $[a_{\beta-\epsilon+1},a_{\beta+\epsilon}]$ at one of its endpoints, say $a_{\beta+\epsilon}$. From there it must continue directly to $a_{\beta-\epsilon}$ (since $[a_{\beta+\epsilon},a_{\beta+\epsilon+1}] \in B$). Now we are stuck, since both boundary edges of the part of $CK(2m-1)$ that is still uncovered by $P$, $[a_{\beta-\epsilon},a_{\beta-\epsilon-1}]$ and $[a_{\beta+\epsilon},a_{\beta+\epsilon+1}]$, belong to $B(\epsilon)$. (The case where the last vertex of $P$ on the edge $[a_{\beta-\epsilon+1},a_{\beta+\epsilon}]$ is $a_{\beta-\epsilon+1}$ is exactly the same.)

\medskip \noindent \textbf{Case II: $h=[a_{\beta},a_{\beta+1}]$ is an internal edge of $P$.} Let $x$ be the vertex that precedes $a_{\beta}$ and $y$ the vertex that follows $a_{\beta+1}$ on $P$. This means that $\langle x,a_{\beta},a_{\beta+1},y \rangle$ is a sub-path of $P$ (see Figure~\ref{fig:opp_dir_broken_spine}). $x$ is not a vertex of the upper $a$-path $\langle a_0,\ldots,a_{\beta} \rangle$. (If it were, then we could not complete $P$ to cover all vertices that lie beyond $[x,a_{\beta}]$, since all boundary edges that lie beyond $[x,a_{\beta}]$ are part of the upper $a$-path, and thus belong to $B(\epsilon)$.) By the same token, $y$ is not a vertex of the lower $a$-path $\langle a_{\beta+1}, \ldots, a_{m-\alpha-\delta} \rangle$.

Since $P$ is a simple path, it follows that both $x$ and $y$ lie on the $b$-path $\langle b_0,\ldots,b_{m-3+\alpha+\delta} \rangle$, and that $x$ precedes $y$ on that path. Moreover, $x$ and $y$ are adjacent on that path, since otherwise $\langle x,a_{\beta},a_{\beta+1},y \rangle$ would separate the remaining vertices into three non-empty parts. Thus $x=b_{\nu}$ and $y=b_{\nu+1}$ for some $\nu$, $0 \leq \nu \leq m-4+\alpha+\delta$.

The next step of the proof will be to split $CK(2m-1)$ into two complete convex geometric graphs $G^+$ and $G^-$ of \textbf{even} order, leaving out just one vertex, in such a way that the edges of $B(\epsilon)$ in $G^+$ [resp., $G^-$] form a blocker for SHP's in $G^+$ [resp., $G^-$]. Here we use the sufficiency part of the characterization of blockers for SHP's in complete convex geometric graphs of even order, given in~\cite{KP16}. To finalize the proof, we shall use the part of $P$ that runs through $G^+$ [or $G^-$] to construct an SHP in $G^+$ [or $G^-$] that avoids a blocker for SHP's in $G^+$ [or $G^-$], and thus reach a contradiction.

Let $G^+$ [$G^-$] be the subgraph of $CK(2m-1)$ of order $2(\alpha+\beta)$ [$2(\gamma+\delta)$] spanned by $V^+$ [$V^-$], where
\[
V^+ = \{a_{\beta},a_{\beta-1},\ldots,a_0,b_0,\ldots,b_{2\alpha+\beta-2}\}
\]
and
\[
V^- = \{b_{2\alpha+\beta},\ldots,b_{m-3+\alpha+\delta},a_{m-\alpha-\delta},a_{m-\alpha-\delta-1},\ldots,a_{\beta+1}\}.
\]
This leaves out the vertex $z=b_{2\alpha+\beta-1}$. The convex hulls of $V(G^+)$ and $V(G^-)$ are disjoint, and their union includes all boundary edges of $CK(2m-1)$, except $h=[a_{\beta},a_{\beta+1}]$ and the two boundary edges incident with $z$.

Recall that the $\alpha$ uppermost beams in $B(\epsilon)$ are $[a_{i_{\nu}},b_{j_{\nu}}]$, $\nu=1,\ldots,\alpha$, where $1 \leq i_1 \leq i_2 \leq \ldots \leq i_{\alpha} \leq \beta-1$, and $j_{\nu}=i_{\nu}+2\nu-2$. Thus
\[
1 \leq i_1=j_1 < j_2 < \ldots < j_{\alpha} \leq \beta-1+2\alpha-2 = 2\alpha+\beta-3<2\alpha+\beta-2,
\]
and hence these $\alpha$ beams are included in $G^+$. The first beam $[a_{i_1},b_{j_1}]$ (if $\alpha>0$) is parallel to $[b_0,a_0]$. The directions of the ensuing $\alpha-1$ diagonals decrease by 2 at a time. This means that these diagonals are parallel to the $\alpha$ boundary edges that precede $[a_0,a_1]$ on the boundary of $G^+$. The $b$-vertices of any two such diagonals are farther apart than their $a$-vertices. Thus, the $\alpha+\beta$ edges of $B(\epsilon)$ in $G^+$ conform to the description of blockers for SHP's in the complete convex geometric graph $G^+$ of order $2(\alpha+\beta)$ (described in detail in~\cite{KP16}). Similar calculations show that the edges of $B(\epsilon)$ in $G^-$, i.e., the $\gamma$ edges of the lower boundary path $\langle a_{\beta+1},\ldots,a_{m-\alpha-\delta} \rangle$ and the $\delta$ lowermost diagonals, conform to the description of a blocker for SHP's in the complete convex geometric graph $G^-$ (of order $2(\gamma+\delta)$).

Now let us return to the short sub-path $\langle x,a_{\beta},a_{\beta+1},y \rangle$ of $P$, where $x=b_{\nu}$ and $y=b_{\nu+1}$ for some $\nu$, $0 \leq \nu \leq m-4+\alpha+\delta$. If $y=z$, as in Figure~\ref{fig:opp_dir_broken_spine}, then the initial part of $P$, ending at $a_{\beta}$ with the edge $[x,a_{\beta}]$ (and excluding $h$, $[a_{\beta+1},y]$, and all subsequent edges of $P$), is an SHP of $G^+$ that avoids all edges of $B$ in $G^+$. But these edges form a blocker for SHP's in $G^+$, a contradiction.

If $y$ lies ``above'' $z$, i.e., if $y$ precedes $z(=b_{2\alpha+\beta-1})$ on the $b$-path, then the initial part of $P$ described above (ending at $a_{\beta}$) covers all vertices of $G^+$ that lie either on or above the diagonal $[x,a_{\beta}]$. We extend this path by adding the diagonal $[a_{\beta},b_{2\alpha+\beta-2}]$, and all boundary edges of the $b$-path from $b_{2\alpha+\beta-2}$ up to $y(=b_{\nu+1})$. This yields an SHP of $G^+$ that avoids all edges of $B$ in $G^+$, a contradiction. (Note that all these additional edges are not in $B$. Indeed, edges of the $b$-path are never in $B$. Beams in $B$ always connect an \emph{internal} vertex of the upper (or lower) part of the $a$-path with a vertex of the $b$-path.)

Finally, if $y$ lies ``below'' $z$, i.e., if $z$ precedes $y$ on the $b$-path, then we reach a contradiction in a similar way, with $G^-$ instead of $G^+$. This completes the proof of the theorem.
\end{proof}

\begin{figure}[tb]
\begin{center}
\scalebox{1.0}{
\includegraphics{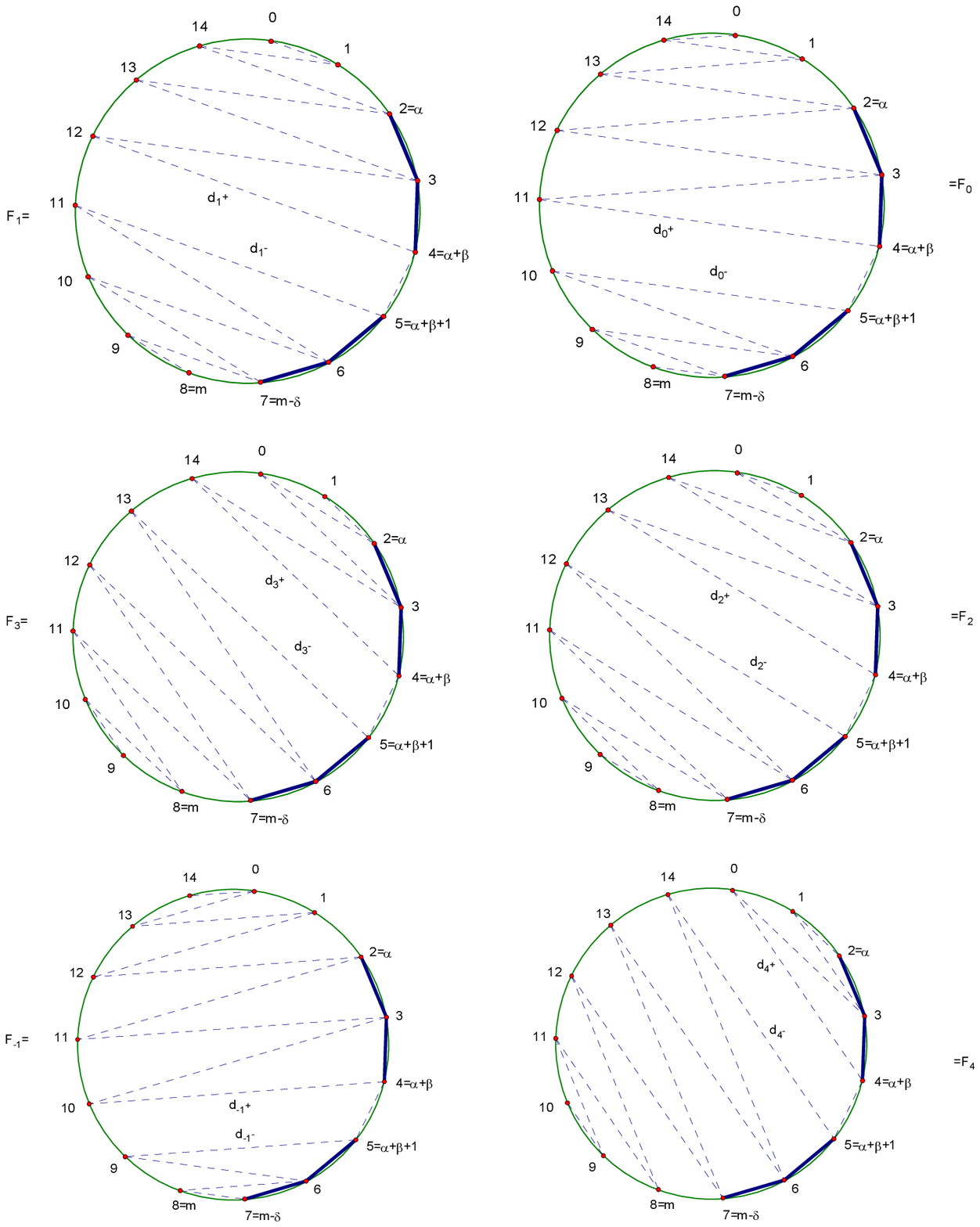}
} \caption{Illustration for the proof of Proposition~\ref{cl:test_paths}. The parameters are $m=8$, $\alpha=2$, $\delta=1$, and $\beta=2$. Depicted in punctured lines are the test paths $F_{-1},F_0,\ldots,F_4$.}
\label{fig:F_i}
\end{center}
\end{figure}

\end{document}